\documentclass[12pt]{amsart}
\usepackage[english]{babel}
\usepackage{egothic}
\usepackage[T1]{fontenc}

\usepackage{amsmath}
\usepackage{amsthm}
\usepackage{amssymb}
\usepackage[all]{xy}
\usepackage{mathrsfs}
\usepackage{enumerate}
\usepackage{physics}
\usepackage[notcite, final, notref]{showkeys}
\usepackage{dsfont}
\usepackage[dvips]{color}
\newtheorem{theorem}{Theorem}[section]

\newtheorem{lemma}[theorem]{Lemma}
\newtheorem{proposition}[theorem]{Proposition}
\newtheorem{corollary}[theorem]{Corollary}
\newtheorem{definition}[theorem]{Definition}

\theoremstyle{definition}
\newtheorem{remark}[theorem]{Remark}

\newtheorem{example}[theorem]{Example}

\usepackage{xcolor}

 \definecolor{ao}{rgb}{0.0, 0.0, 1.0}
\definecolor{applegreen}{rgb}{0.55, 0.71, 0.0}
\definecolor{cadmiumorange}{rgb}{0.93, 0.53, 0.18}
\definecolor{anti-flashwhite}{rgb}{0.95, 0.95, 0.96}
 	\definecolor{ao(english)}{rgb}{0.0, 0.5, 0.0}

\title[]{Short-time Fourier transform and superoscillations}

\author[D. Alpay]{Daniel Alpay}
\address{(DA) Schmid College of Science and Technology \\
Chapman University\\
One University Drive
Orange, California 92866\\
USA}
\email{alpay@chapman.edu}

\author[A. De Martino]{Antonino De Martino}
\address{(ADM) Politecnico di
Milano\\Dipartimento di Matematica\\Via E. Bonardi, 9\\20133 Milano\\Italy}
\email{antonino.demartino@polimi.it}

\author[K. Diki]{Kamal Diki}
\address{(KD) Schmid College of Science and Technology \\
Chapman University\\
One University Drive
Orange, California 92866\\
USA}
\email{diki@chapman.edu}

\author[D. C. Struppa]{Daniele C. Struppa}
\address{(DCS) Schmid College of Science and Technology \\
Chapman University\\
One University Drive
Orange, California 92866\\
USA}
\email{struppa@chapman.edu}

\begin{document}
\maketitle
\begin{abstract}
	In this paper we investigate new results on the theory of superoscillations using time-frequency analysis tools and techniques such as the short-time Fourier transform (STFT) and the Zak transform.
	We start by studying how the short-time Fourier transform acts on superoscillation sequences. We then apply the supershift property to prove that the short-time Fourier transform preserves the superoscillatory behavior by taking the limit. It turns out that these computations lead to interesting connections with various features of time-frequency analysis such as Gabor spaces, Gabor kernels, Gabor frames, 2D-complex Hermite polynomials, and polyanalytic functions. We treat different cases depending on the choice of the window function moving from the general case to more specific cases involving the Gaussian and the Hermite windows. We consider also an evolution problem with an initial datum given by superoscillation multiplied by the time-frequency shifts of a generic window function. Finally, we compute the action of STFT on the approximating sequences with a given Hermite window.
\end{abstract}



\tableofcontents

\section{Introduction}
\setcounter{equation}{0}

This paper is dedicated to the study of the relationship between an important tool in Fourier Analysis, namely
the \textit{short-time Fourier transform (STFT)}, and an interesting phenomenon that has been recently the subject of much study, namely the superoscillatory behavior of certain sequences of exponentials, a phenomenon often subsumed under the name of {\it superoscillations}.

We will discuss the meaning of both these terms in more detail later, but for now it will suffice to say that STFT is a transformation that is associated to a chosen {\it window function}, which allows the computation of the spectrogram of a given signal, taking into account both time and frequency; as a result, this integral transform maps isometrically the Hilbert space $L^2(\mathbb{R}^d)$ to $L^2(\mathbb{R}^{2d})$  modulo the norm of the fixed window function. In \cite{Asampling} the STFT was used to study sampling problems and extend the Segal-Bargmann theory to the polyanalytic setting. These extensions in the case of metaanalytic functions were considered in \cite{FLKC2014}. The authors of \cite{ABSH2023} used the properties of STFT and 2D-complex Hermite polynomials to compute the spectrogram of white noise with Hermite windows. In the quaternionic and Clifford settings various counterparts of the short-time Fourier transform were developed recently in \cite{DMA2022, DMD2021, DMD2022}. We refer to the book \cite{GK2001} and the references therein for further discussions on the STFT and related topics in time-frequency analysis.
\\ \\
The theory of superoscillations is a concept that arose in quantum mechanics in connection with the Aharonov theory of weak measurements, \cite{aav}. Mathematically, it refers to the combination of small Fourier components, whose spectrum is bounded, resulting in a shift that can be much larger than the spectrum itself. The mathematics of superoscillations and its different applications in physics was developed by Aharonov and his collaborators, see \cite{ACSSTbook2017, ACSST2011JPA, BZACSSTRQHL2019} and the references therein. The authors of \cite{CSSY2023} studied new generating functions associated to the superoscillatory coefficients leading to interesting connections with families of special functions such as Bernstein and Hermite polynomials. Furthermore, new integral representations of superoscillations and related results were obtained in the recent paper \cite{BCSS2023}. \\ \\ The phenomena of superoscillations can be connected to the theory of Fock spaces thanks to the use of the Segal-Bargmann transform, see \cite{ACDSS}. Actually, the Segal-Bargmann transform can be obtained as a particular example of the short-time Fourier transform corresponding to the Gaussian window, see \cite[Proposition 3.4.1]{GK2001}. Therefore, inspired by this observation we investigate how STFT can be used to study superoscillations.  To this end, we will start by computing the action of STFT on superoscillations and prove new integral representations as a consequence of these calculations. We then develop new research further connecting time-frequency analysis to the theory of superoscillations.
\\ \\

The structure of the paper is as follows. In Section 2 we review and collect basic notions on the STFT, translation-modulation operators, and Gabor and Fock spaces. We also recall the basic notations and definitions about superoscillations. In Section 3 we compute the action of STFT on superoscillations in the case of a generic window and prove the connections with Gabor kernels. We also obtain a new integral representation of superoscillations thanks to the use of the reconstruction formula. In Section 4 and Section 5 we treat respectively the cases where the window function is given by the Gaussian and Hermite functions. In these two sections we further investigate additional  relations among Fock spaces, polyanalytic Fock spaces, and 2D-complex Hermite polynomials. In Section 6 we use the Zak transform to prove new results on Gabor frames involving superoscillations both in the case of Gaussian and Hermite windows. Then, in Section 7 we study the Cauchy problem for the Schrödinger equation when the initial datum is given by considering the action of the translation-modulation operators on a generic window function in $L^2(\mathbb{R})$. Finally, in Section 8 we study how STFT acts on the approximating sequences. The results suggest an interesting research question on polyanalytic supershift property.

\section{Preliminary results}
\setcounter{equation}{0}
In this section we review basic properties and well-known facts that will be useful in the sequel. For further details about the topics considered here we refer to the books \cite{ACSSTbook2017, GK2001} and the references therein. \\ \\
\textbf{Translation-modulation operators and Fourier transform:}\\ \\ For a given time-domain signal $f:\mathbb{R} \to \mathbb{C}$, and $x,\omega\in \mathbb{R}$ fixed, we define the translation and modulation operators by:
\begin{equation}
(T_xf)(t)=f(t-x), 
\end{equation}
and \begin{equation}
M_\omega f(t)=e^{i\omega t}f(t).
\end{equation}
The two operators satisfy the so-called \textit{canonical commutation rule }, given by
\begin{equation}
\label{commu}
T_xM_\omega=e^{- i x\omega}M_\omega T_x.
\end{equation}
As a consequence of the previous identity we observe that the translation $T_x$ and the modulation $M_\omega$ operators commute if and only if $ x \omega \in 2\pi \mathbb{Z}$. The operators $T_xM_\omega$ and $M_\omega T_x$ are called time-frequency shifts. We recall also that the Fourier transform of a signal $f \in L^1(\mathbb{R}) \cap L^2(\mathbb{R})$ is given by
\begin{equation}
\mathcal{F}(f)(\lambda)= \hat{f}(\lambda):=\int_{\mathbb{R}}e^{- it \lambda}f(t)dt, \quad \forall \lambda\in \mathbb{R}.
\end{equation}
The action of the Fourier transform on the translation and modulation is given by
\begin{equation}
\label{basic}
\mathcal{F}(T_xf)=M_{-x}\mathcal{F}(f), \qquad \mathcal{F}(M_{x}f)=T_{x}\mathcal{F}(f).
\end{equation}
These relations imply that the action of the Fourier transform on the time-frequency shifts is given by
\begin{equation}
\label{ts}
\mathcal{F}(M_{\omega}T_x f)=T_{\omega}M_{-x} \mathcal{F}(f).
\end{equation}
\\ \\
\textbf{Basic properties of the short-time Fourier transform (STFT)} \\ \\

Let us fix a function $g\neq 0$, $g\in L^2(\mathbb{R})$, called \textit{window function}. The STFT of a given signal $f \in L^2(\mathbb{R})$ with respect to the window function $g$ is given by 
\begin{equation}
\label{SFT}
V_g (f)(x,\omega):=\int_{\mathbb{R}} e^{-i t\omega}\overline{g(t-x)}f(t)dt, 
\end{equation}
for every $x,\omega \in \mathbb{R}$. In other fields, such as Physics, this integral transform is known as Gabor transform, windowed Fourier transform, time frequency representation or even coherent-state representation, see \cite{A, GA}. It is also possible to introduce the STFT as the inner product of the original signal $f$ with the time-frequency shift operator $M_\omega T_x$ on the Hilbert space $L^2(\mathbb{R})$ as follows
\begin{equation}
\label{inner}
V_g(f)(x,\omega)=\langle f, M_\omega T_x g \rangle_{L^2(\mathbb{R})}, 
\end{equation}
for every $x,\omega \in \mathbb{R}$ and $f\in L^2(\mathbb{R})$.
The STFT defines an isometric operator mapping the space $L^2(\mathbb{R})$ into the space $L^2(\mathbb{R}^2)$ so that we have 
\begin{equation}
\label{moyal}
||V_g (f)||_{L^2(\mathbb{R}^2)}=||f||_{L^2(\mathbb{R})}||g||_{L^2(\mathbb{R})},
\end{equation}
the measure on $\mathbb{R}^2$, in the above formula, is the Lebesgue measure.
\\More generally, if $g_1$ and $g_2$ are two window functions in $L^2(\mathbb{R})$ and $f_1$ and $f_2$ are two signals in $L^2(\mathbb{R})$, then the STFT satisfies the so-called Moyal formula  given by 
\begin{equation}
\langle V_{g_1}(f_1), V_{g_2} (f_2)\rangle_{L^2(\mathbb{R}^2)}=\langle f_1, f_2 \rangle_{L^2(\mathbb{R})} \overline{\langle g_1, g_2 \rangle}_{L^2(\mathbb{R})}.
\end{equation}

Moreover, the action of the STFT on the time-frequency shifts of the signal $f$ is called \emph{covariance property}, and it is given by

\begin{equation}
\label{cov}
V_g(T_u M_{\eta} f)(x, \omega)= e^{- i u \omega} V_g f(x-u, \omega-\eta).
\end{equation}

The reconstruction formula allows to reconstruct the original signal $f$ using the STFT and the window function $g$, and it is given by 
\begin{equation}
\label{inversion}
f(y)=\frac{1}{||g||_{L^2(\mathbb{R})}^2}\int_{\mathbb{R}^2}V_g(f)(x,\omega)(M_\omega T_x g)(y)dxd\omega, \quad \forall y\in \mathbb{R}.
\end{equation} 

For a window function $g$ such that $\| g \|_{L^2(\mathbb{R})}=1$ we define the spectrogram of the signal $f$ with respect to the window $g$ to be the modulus square of the STFT. Indeed, the spectrogram is defined by the following mathematical expression
\begin{equation}
	SPEC_g(f)(x,\omega):=|V_g f(x,\omega)|^2, \quad \forall x, \omega\in \mathbb{R}.
\end{equation}
The STFT with window function given by the Gaussian $ \varphi(t)=e^{- \frac{ t^2}{2}}$ is related to the Bargmann transform, see \cite{Barg, GK2001}. This integral transform is defined as
\begin{equation}
\label{kernelb}
 \mathcal{B}(f)(z)= \int_{\mathbb{R}} A(z,t)  f(t) dt, \qquad A(z,t):=\left(\frac{1}{\pi}\right)^{\frac{3}{4}}e^{- \frac{(z^2+t^2)}{2}+\sqrt{2} zt}.
\end{equation}
The relation between the STFT and the Bargmann transform is given by
$$V_{\varphi}(f)(x,- \omega)= \pi^{\frac{3}{4}} e^{- \frac{|z|^2}{2}+\frac{ix \omega}{2}} \mathcal{B}(f)(z), \qquad z= \frac{x+i \omega}{\sqrt{2}}.$$
see \cite[Prop. 3.4.1]{GK2001}. The Bargmann transform is an unitary integral transform that maps the space $L^{2}(\mathbb{R})$ into the Fock space $ \mathcal{F}(\mathbb{C})$. This space is a subspace of entire functions, and it is defined by
$$ \mathcal{F}(\mathbb{C}):= \left\{f \in H(\mathbb{C}), \qquad \| f\|_{\mathcal{F}(\mathbb{C})}^2=  \int_{\mathbb{C}} | f(z)|^2 e^{-|z|^2} d \lambda(z) \right\},$$
where $d \lambda(z)=dx dy$ denotes the classical Lebesgue measure with $z=x+i y$.
\\The reproducing kernel of the Fock space is given by the function
$$ K(z,w)= \frac{1}{\pi} e^{z \bar{w}}, \qquad z, w \in \mathbb{C}.$$
The normalized  reproducing kernel of the space $ \mathcal{F}(\mathbb{C})$ is defined as
$$ k_w(z)=\frac{K(z,w)}{\| K_w \|_{\mathcal{F}(\mathbb{C})}}= \frac{1}{\sqrt{\pi}}e^{z \bar{w}- \frac{|w|^2}{2}}.$$
We observe also that we can write the following relation between the kernel of the Bargmann transform and the reproducing kernel of the Fock space
\begin{equation}
\label{trick}
\langle A_z, A_w \rangle_{L^2(\mathbb{R})}= \frac{1}{\pi} e^{z \bar{w}},
\end{equation}
where by the notation $A_z$, $A_w$ we mean the Bargmann kernel $A(z,x)$ where the variable $x$ is fixed.
Now, we define Hermite functions as
\begin{equation}
\label{herm}
h_n(t):= e^{-\frac{t^2}{2}}H_n(t)=(-1)^n e^{\frac{t^2}{2}} \frac{d^n}{d t^n}(e^{-t^2}), \qquad n=0,1,...
\end{equation}
where $H_n(t)$ are the Hermite polynomials. It turns out that the Segal-Bargmann transform maps the normalized Hermite functions onto an orthonormal basis of the Fock space. Precisely for $n=0,1,...$ we have
\begin{equation}
\label{action}
\mathcal{B}(h_n)(z)= \left(\frac{1}{\pi}\right)^{\frac{1}{4}} 2^{\frac{n}{2}} z^n, \qquad \forall z \in \mathbb{C}.
\end{equation}  
If we pick as window function of the STFT the Hermite functions we get a link with the so-called polyanalytic Bargmann space, see \cite{Asampling}. \\ 

\textbf{Structure of the Gabor spaces} \\ 

The Gabor space, denoted by $\mathcal{G}_g$, is the subspace of $L^{2}(\mathbb{R}^2)$ defined as the image of $L^2(\mathbb{R})$ under the STFT with window function $g$, i.e.,
$$\mathcal{G}_g:=\lbrace{V_g(f), \quad f\in L^2(\mathbb{R})}\rbrace.$$ 
It turns out the Gabor space is a reproducing kernel Hilbert space whose reproducing kernel is given by the Gabor kernel defined by 
\begin{equation}
\label{gaborkernel}
K_g(x,\omega;u, \eta):=\langle M_\omega T_x g, M_\eta T_u g \rangle_{L^2(\mathbb{R})},
\end{equation}
for every $(x,\omega), (u,\eta)\in \mathbb{R}^2$, see \cite{D}. If we consider the Gaussian $\varphi(x)= e^{- \frac{x^2}{2}}$ the reproducing kernel of the Gabor space can be written as
\begin{equation}
\label{Gaborg}
K_{\varphi}(x, \omega; u, \eta)= \sqrt{\pi} e^{\frac{i}{2}(u+x)(\omega- \eta)} e^{- \frac{(u-x)^2}{4}- \frac{ (\eta- \omega)^2}{4}}.
\end{equation}
The kernel $K_g$ takes another peculiar form if we pick as a window function $g$ the Hermite functions, see \cite{A, AF}. Precisely, we have
\begin{equation}
\label{Gaborh}
K_{h_n}(x, \omega; u, \eta)= \sqrt{\pi} e^{\frac{i}{2}(u+x)(\omega- \eta)} e^{- \frac{(u-x)^2}{4}- \frac{ (\eta- \omega)^2}{4}} L_n^0 \left( \frac{(x-u)^2+(\omega- \eta)^2}{2}\right),
\end{equation}
where $L_{n}^0$ are the Laguerre polynomials defined by
\begin{equation}
\label{laguerre}
 L_n^{0}(x)= \sum_{i=0}^{n} (-1)^i \binom{n}{n-i} \frac{x^i}{i!}.
\end{equation}
\newline

\textbf{The Weyl-Heisenberg symmetry and the Bargmann-Fock representation}
\\ \\
We follow  \cite{GK2001} to discuss an interesting connection relating the STFT to the Heisenberg group and its Schrödinger and Bargmann-Fock representations. See also \cite{F}. The Heisenberg group $\mathbb{H}:=\mathbb{R}^{2}\times \mathbb{R}\simeq \mathbb{C}\times\mathbb{R}$ is a group with respect to the law given by 
\begin{equation}
(x,\omega, \theta)\cdot (x',\omega', \theta')=(x+x',\omega+\omega',\theta+\theta'+\frac{1}{2}(x'\omega-x\omega')),
\end{equation}
for every $(x,\omega, \theta), (x',\omega', \theta')\in \mathbb{H}.$ If we use the identification of $(x,\omega),(x',\omega')\in \mathbb{R}^2$ with the complex numbers $z=x+i\omega, z=x'+i\omega' \in \mathbb{C}$ the law of the Heisenberg group can be expressed as follows

$$(z,\theta)\cdot (z',\theta')=(z+z',\theta+\theta'+\frac{1}{2}\Im(z\overline{z'})).$$

 The Schrödinger representation is a representation of the Heisenberg group by time-frequency shifts. In fact, the Schrödinger representation $\rho:\mathbb{H}\longrightarrow \mathcal{U}(L^2(\mathbb{R}))$ maps the Heisenberg group into unitary operators of $L^2(\mathbb{R})$. More precisely, it is defined by

\begin{equation}
\rho(x,\omega,\theta):=e^{i\theta}e^{\frac{ix\omega}{2}}T_xM_\omega,
\end{equation}

for every $(x,\omega,\theta) \in \mathbb{H}$ where $\mathcal{U}(L^2(\mathbb{R}))$ denotes the space of unitary operators of $L^2(\mathbb{R})$. The STFT is connected to the Schrödinger representation via the following formula

\begin{equation}
\langle f, \rho(x,\omega,\theta)g  \rangle_{L^2(\mathbb{R})}=e^{-i\theta +\frac{ix \omega}{2}} V_g(f)(x, \omega), \quad \forall (x,\omega,\theta)\in \mathbb{H}.
\end{equation}

 Thanks to the identification with complex numbers we can consider the so-called Bargmann-Fock representation $\beta:\mathbb{H}\longrightarrow \mathcal{U}(\mathcal{F}(\mathbb{C}))$ which maps the Heisenberg group $\mathbb{H}$ into unitary operator of the Fock space. Let us fix $(z,\theta)\in \mathbb{H}$ in the Heisenberg group with the the identification $z=x+i\omega\simeq (x,\omega)$. Then the construction of the Bargmann-Fock representation is justified by the following commutative diagram

$$\xymatrix{
    \mathcal{F}(\mathbb{C}) \ar[r]^{\beta(z,\theta)} \ar[d]_{\mathcal{B}^{-1}} & \mathcal{F}(\mathbb{C}) \\ L^2(\mathbb{R}) \ar[r]_{\rho(x,\omega,\theta)} & L^2(\mathbb{R}) \ar[u]_{\mathcal{B}}
  }$$ 
  Moreover, the Bargmann-Fock representation $\beta(z,\theta)$ can be computed via the Bargmann transform as follows
  \begin{equation}
\beta(z,\theta)=\mathcal{B}\rho(x,\omega,\theta)\mathcal{B}^{-1}.
\end{equation}
  
 From the discussion above it turns out that the Schrödinger representation and the Bargmann-Fock representation are two equivalent representations of the Heisenberg group that are connected via the Bargmann transform. It turns out that the Bargmann-Fock representation obtained by developing corresponds to the Weyl operator on the Fock space. See \cite[Formula 9.18]{GK2001} and \cite{zhu} for further explanations of this approach and how it can be related to the STFT and the Weyl operators on the Fock space.
 \begin{remark}
The Weyl-Heisenberg symmetry and the Bargmann-Fock representation will be fundamental in Section 5.2.
 \end{remark}

\textbf{Mathematics of superoscillations}
\\ \\
We will now review the basic principles of superoscillations: for more details we refer the reader to \cite{ACSSTbook2017}. The prototypical superoscillating function, that appears in the theory of weak values,  is
\begin{equation}\label{FNEXP}
F_n(t,a)
=\sum_{j=0}^nC_j(n,a)e^{i(1-\frac{2j}{n})t},\ \ t\in \mathbb{R},
\end{equation}
where $a>1$ and the coefficients $C_j(n,a)$ are given by
\begin{equation}\label{Ckna}
C_j(n,a)=\binom{n}{j}\left(\frac{1+a}{2}\right)^{n-j}\left(\frac{1-a}{2}\right)^j.
\end{equation}

If we fix $t \in \mathbb{R}$  and we let $n$ tends to infinity, we  obtain that
\begin{equation}
\label{limit}
\lim_{n \to \infty} F_n(t,a)=e^{iat},
\end{equation}
and the limit is uniform on compact subsets of the real line. The term superoscillations comes from the fact that
in the Fourier representation of the function \eqref{FNEXP} the frequencies
$1-2j/n$ are bounded by 1, but the limit function $e^{iat}$ has a frequency $a$ that can be arbitrarily larger  than $1$. Inspired by this example we define a {\em generalized Fourier sequence}. These are sequences of the form

\begin{equation}\label{basic_sequenceq}
f_n(t,a):= \sum_{j=0}^n Z_j(n,a)e^{ih_j(n)t},\ \ \ n\in \mathbb{N},\ \ \ t\in \mathbb{R},
\end{equation}
where $a\in\mathbb R$, $Z_j(n,a)$ and $h_j(n)$
are complex and real-valued functions, respectively.
The sequence \eqref{basic_sequenceq}
 is said to be {\em a superoscillating sequence} if
 $\displaystyle\sup_{j,n}|h_j(n)|\leq 1$ and
 there exists a compact subset of $\mathbb R$,
 which will be called {\em a superoscillation set},
 on which $f_n(t)$ converges uniformly to $e^{ig(a)x}$,
 where $g$ is a continuous real-valued function such that $|g(a)|>1$.

In superoscillations we see how the value of the exponential function $x \to e^{ixt}$ for large values of $x$ can be achieved as a limit of combinations of values of that same functions on a countable set of small values of $x$. This idea is a special case of a much more general concept, known as {\it supershift}.

\begin{definition}[Supershift Property]\label{supershift}
 Let $\lambda\mapsto \varphi_\lambda(X)$ be a continuous complex-valued function
in the variable $\lambda \in \mathcal{I}$, where $\mathcal{I}\subseteq\mathbb{R}$ is an interval, and
$X\in \Omega $, where $\Omega$ is a domain.
We consider  $X\in \Omega $ as a parameter for the function
 $\lambda\mapsto \varphi_\lambda(X)$ where $\lambda \in \mathcal{I}$.
When  $[-1,1]$ is contained in $\mathcal{I}$ and $a\in \mathcal{I}$, we define the sequence
\begin{equation*}
\label{psisuprform}
\psi_n(X)=\sum_{j=0}^nC_j(n,a)\varphi_{1-\frac{2j}{n}}(X),
\end{equation*}
in which $\varphi_{\lambda}$ is computed just at the points $1-\frac{2j}{n}$ which belong to the interval $[-1,1]$ and the coefficients
$C_j(n,a)$ are defined for example as in \eqref{Ckna}, for $j=0,...,n$ and $n\in \mathbb{N}$.
If
$$
\lim_{n\to\infty}\psi_n(X)=\varphi_{a}(X)
$$
for $|a|>1$ arbitrary large (but belonging to $\mathcal{I}$), we say that the function
$\lambda\mapsto \varphi_{\lambda}(X)$,  for $X$ fixed, admits the supershift property.
\end{definition}

 \begin{remark}\label{anal-supershift}
Note that if the function $\varphi_\lambda (X)$ in Definition \ref{supershift} is analytic with respect to $\lambda$ as a complex variable, then the function $\lambda\mapsto \varphi_\lambda(X)$ admits the supershift property. This fact can be seen as a direct consequence of the result proved in \cite[Theorem 4.8]{ACJSSST2022} in the case of two variables.
 \end{remark}

We recall an interesting technical lemma which will be useful in the sequel, see \cite[formula 3.323 (2)]{GR}.

\begin{lemma}[Gaussian integral]
\label{gaussint}
Let $ \alpha>0$ and $w \in \mathbb{C}$. Then
\begin{equation}
\label{gint}
\int_{\mathbb{R}} e^{- \alpha t^2+wt}dt= \left(\frac{\pi}{\alpha}\right)^{1/2} e^{\frac{w^2}{4 \alpha}}.
\end{equation}
\end{lemma}

\section{STFT and super-oscillations: Generic window}
\setcounter{equation}{0} In this section we study how the STFT acts on superoscillations and, by using STFT with a generic window $g \in L^2(\mathbb{R})$, we prove a new integral representation for superoscillatory sequence. 

\begin{definition}
\label{sig}
Let $a>1$, $n \in \mathbb{N}$, $x\in \mathbb{R}$ and $g\in L^2(\mathbb{R})$ being a fixed window. We define the signal-superoscillation function in time-domain by setting
\begin{equation}
\label{signalg}
S^{g,x}_{n}(t,a):=F_n(t,a)(T_xg)(t), \quad \forall t\in \mathbb{R}.
\end{equation}

\end{definition}

\begin{remark}
Since the window function $g$ belongs to the Hilbert space $L^2(\mathbb{R})$, then also the signal function $S^{g,x}_{n}$ belongs to $L^2(\mathbb{R})$. Hence, it makes sense to compute the STFT of the signal \eqref{signalg}.
\end{remark}
\begin{example}
The class of functions $S^{g}_{n}(t,a)$ includes the superoscillation-wave functions studied recently in \cite{ACDSS}. Indeed, for $x=0$ we give these two examples: 
\begin{enumerate}
\item[i)] If $g(t)=e^{-\frac{t^2}{2}}$ we have
$$S^{g,0}_{n}(t,a)=e^{-\frac{t^2}{2}}F_n(t,a), \quad \forall t\in \mathbb{R}.$$
\item[ii)] If $g(t)=\frac{h_k(t)}{\| h_{k}\|_{L^2}}$ are the normalized Hermite functions. Then, for any $t \in \mathbb{R}$ we have
$$S^{g,0}_{n}(t,a)=h_k(t)F_n(t,a), \qquad k=1,2, \cdots$$
\end{enumerate}

\end{example}
The STFT of the signal function \eqref{signalg} is a linear combination of reproducing kernels of the Gabor space. More precisely, we have the following result.
\begin{theorem}
\label{main}
Let $a>1$ and $x\in \mathbb{R}$. Let us consider $g\in L^2(\mathbb{R})$ being a fixed window function.  
For every $j=0,1,\cdots, n$ we set $\omega_j=1-\frac{2j}{n}$. Then
\begin{equation}
\displaystyle V_g\left(S^{g,x}_{n}\right)(u,\eta)=\sum_{j=0}^{n}C_j(n,a)K_g(x,\omega_j; u,\eta),
\end{equation}
for every $(u,\eta)\in \mathbb{R}^2$. In other terms, we have 
\begin{equation}
\displaystyle V_g\left(S^{g,x}_{n}\right)(u,\eta)=\sum_{j=0}^{n}C_j(n,a) \langle M_{w_j}T_x g, M_\eta T_u g \rangle_{L^2(\mathbb{R})}.
\end{equation}
\end{theorem}
\begin{proof}
First, we observe that
\begin{align*}
\displaystyle S^{g,x}_{n}(t,a)&=F_n(t,a)(T_xg)(t)\\
&=\sum_{j=0}^{n}C_j(n,a)e^{it(1-2j/n)}(T_xg)(t)\\
&=\sum_{j=0}^{n}C_j(n,a)M_{\omega_j}(T_xg)(t).\\
\end{align*}
Therefore, combining the previous equality with the definition of STFT, see \eqref{inner}, we get  
\begin{align*}
V_g\left(S^{g,x}_{n}\right)(u,\eta)&=\langle S^{g,x}_{n}, M_\eta T_x g \rangle_{L^2(\mathbb{R})}\\
&=\sum_{j=0}^{n}C_j(n,a)\langle M_{\omega _j}T_x g, M_\eta T_u g \rangle_{L^2(\mathbb{R})}\\
&=\sum_{j=0}^{n}C_j(n,a)K_g(x,\omega_j;u,\eta).
\end{align*}
This ends the proof.
\end{proof}
From the limit property of the superoscillation sequence $F_n(t,a)$, see \eqref{limit}, we have
 \begin{equation}
 \label{signal}
  \lim_{n\longrightarrow \infty}S^{g,x}_{n}(t,a)=e^{iat}T_xg(t)=(M_{a}T_x g)(t):=S^{g,x}(t,a), \quad \forall t\in \mathbb{R}.
 \end{equation}
As a consequence of the supershift property we can prove that the superoscillatory behavior is preserved by the short-time Fourier transform. In fact, the action of STFT for the limit function $S^{g,x}(t,a)$ is given by the following result.
 \begin{proposition}
 \label{limitcase}
 It holds that 
 \begin{equation}
 V_g(S^{g,x})(u,\eta)=K_g(x,a; u,\eta),
 \end{equation}
 for every $(x, u,\eta)\in \mathbb{R}^3$.
 \end{proposition}
 \begin{proof}
By definition of the STFT applied to the signal $S^{g,x}$ and the definition of the Gabor kernel, see \eqref{gaborkernel}, we get
\begin{eqnarray*}
V_g(S^{g,x})(u,\eta)&=& \langle M_{a} T_x g, M_{\eta} T_{u} g \rangle_{L^2(\mathbb{R})}\\
&=&K_g(x,a; u,\eta).
\end{eqnarray*}
 \end{proof}

 As a consequence of the previous calculations and the supershift property we can prove the following result.
\begin{theorem}
Let $a>1$ and $g\in L^2(\mathbb{R})$. Then it holds that 
\begin{equation}
\label{limit2}
\lim_{n\longrightarrow \infty} V_g(S^{g,x}_{n})(u,\eta)=K_g(x,a; u,\eta),
\end{equation}
for every $(x, u,\eta)\in \mathbb{R}^3$. Moreover, if $\omega_j=1- \frac{2j}{n}$, we have
\begin{equation}
\label{change}
	\lim_{n\longrightarrow \infty} \sum_{j=0}^{n}C_j(n,a)K_g(x,\omega_j;u,\eta) =K_g(x,a; u,\eta), \quad \forall (x, u,\eta)\in \mathbb{R}^3.
\end{equation}

\end{theorem}
\begin{proof}
We start by rewriting the Gabor kernel (see \eqref{gaborkernel}) in its integral form
	$$ K_g \left(x, a; u, \eta \right)= \int_{\mathbb{R}} e^{ i t\left(a- \eta\right)} g(t-x) \overline{g(t-u)} dt.$$
It is clear that this function is entire in the variable $a$. Then by combining the supershift property (see Definition \eqref{supershift}), Theorem \ref{main}, and Proposition \ref{limitcase}, we get the result. 
\\ Formula \eqref{change} follows by writing the explicit expression of $V_g(S^{g,x}_{n})$, see Theorem \ref{main}, in \eqref{limit2}.
\end{proof}

As a direct application of the reconstruction formula for STFT we obtain the following integral representation of  the superoscillation sequences.

\begin{theorem}[STFT integral representation]
\label{intrep}
Let $a>1$, $x \in \mathbb{R}$ being fixed, and let $g\in L^2(\mathbb{R})$ with $g\neq 0$. Then, the superoscillating sequence $F_n(y,a)$ satisfies the following integral representation 
\begin{equation}
F_n(y,a)=\frac{1}{g(y-x)||g||^2_{L^2(\mathbb{R})}}\int_{\mathbb{R}^2}\varphi_n(u,\eta)(M_\eta T_ug)(y)dud\eta, 
\end{equation}
for every $y \in \mathbb{R}$, where we have set 
\begin{equation}
\label{phi}
\varphi_n(u,\eta)=\sum_{j=0}^{n}C_j(n,a)K_g(x,\omega_j;u,\eta), \quad \text{and } \quad \omega_j=1- \frac{2j}{n}.
\end{equation}

\end{theorem}

\begin{proof}
Since $S^{g,x}_{n}(t,a) \in L^2(\mathbb{R})$ we can apply formula \eqref{inversion}. Thus we get
$$ S^{g,x}_{n}(y,a)= \frac{1}{\| g \|^2_{L^2(\mathbb{R})}} \int_{\mathbb{R}^2} V_{g}(S^{g,x}_{n})(u, \eta) (M_\eta T_ug)(y) du d \eta.$$
By Theorem \ref{main} we get
$$ S^{g,x}_{n}(y,a)= \frac{1}{\| g \|^2_{L^2(\mathbb{R})}} \int_{\mathbb{R}^2} \varphi_n(u, \eta) (M_\eta T_ug)(y) du d \eta.$$
Finally, from the definition of the function $S^{g,x}_{n}(y,a)$ we get the thesis.

\end{proof}
\begin{remark}
The result proved in Theorem \ref{intrep} can be reformulated if we replace the superoscillating sequence $F_n(y,a)$ by the generalized Fourier sequence
$$
f_n(y,a):= \sum_{j=0}^n Z_j(n,a)e^{ih_j(n)y},\ \ \ n\in \mathbb{N},\ \ \ y \in \mathbb{R},
$$
where $a\in\mathbb R$, $Z_j(n,a)$ and $h_j(n)$ are complex and real-valued functions, respectively. In particular, we observe that the generalized Fourier sequence $f_n(y,a)$ satisfies the following integral representation 
\begin{equation}
f_n(y,a)=\frac{1}{g(y-x)||g||^2_{L^2(\mathbb{R})}}\int_{\mathbb{R}^2}\psi_n(u,\eta)(M_\eta T_ug)(y)dud\eta, 
\end{equation}
for every $y\in \mathbb{R}$, where we have set 
$$
\psi_n(u,\eta)=\sum_{j=0}^{n}Z_j(n,a)K_g(x,h_j(n);u,\eta).
$$

\end{remark}

Now, we compute the $L^2$-norm of $V_{g}$ applied to the signal \eqref{signalg}.

\begin{lemma}
\label{norm4}
Let $a>1$ and $x \in \mathbb{R}$. Let us consider $g \neq 0$ and $g \in L^2(\mathbb{R})$. Then we have
\begin{equation}
\label{norm5}
\| V_g (S^{g,x}_n) \|_{L^2(\mathbb{R}^2)}^2= \| g \|_{L^2(\mathbb{R})}^2 \sum_{j,k=0}^n C_{j}(n,a)C_k(n,a) \mathcal{F}(\psi_{g,x}) \left(\frac{2(j-k)}{n}\right),
\end{equation}
where $\psi_{g,x}(t)=|(T_x g)(t)|^2$.
\end{lemma}
\begin{proof}
To prove the result we need formula \eqref{moyal}. Hence we compute first the norm of the signal $S^{g,x}_{n}$. From the definition of the $L^2$-norm we have
\begin{eqnarray*}
\|S^{g,x}_{n} \|_{L^{2}(\mathbb{R})}^2&=& \int_{\mathbb{R}}| (T_x g)(t)| | F_n(t,a)|^2 dt\\
&=& \int_{\mathbb{R}}| g(t-x)|^2 | F_n(t,a)|^2 dt.
\end{eqnarray*}
By the definition of the superoscillating sequence, see \eqref{FNEXP}, we obtain

$$ |F_n(t,a)|^2= \sum_{j,k=0}^n C_j(n,a)C_k(n,a) e^{\frac{2i t}{n}(k-j)}.$$

Then by the definition of the Fourier transform it follows that
\begin{eqnarray}
\nonumber
\|S^{g,x}_n \|_{L^{2}(\mathbb{R})}^2&=& \sum_{j,k=0}^n C_j(n,a)C_k(n,a) \int_{\mathbb{R}} e^{- 2 i t \frac{(j-k)}{ n}} | g(t-x)|^2 dt\\
\label{star}
&=& \sum_{j,k=0}^n C_{j}(n,a)C_k(n,a) \mathcal{F}(\psi_{g,x}) \left(\frac{2(j-k)}{n}\right),
\end{eqnarray}

where we set $  \psi_{g,x}(t)=|g(t-x)|^2=|(T_x g)(t)|^2$. Finally by plugging \eqref{star} into
$$ \| V_g (S^{g,x}_n) \|_{L^2(\mathbb{R}^2)}^2= \| g \|_{L^2(\mathbb{R})}^2 \|S^{g,x}_n \|_{L^{2}(\mathbb{R})}^2,$$
we have the result.
\end{proof}

\begin{remark}
 Since the function $g \in L^2(\mathbb{R})$ and the space $L^2(\mathbb{R})$ is invariant under translation it is clear that $\psi_{g,x} \in L^{1}(\mathbb{R})$. This implies that $ \mathcal{F}(\psi_{g,x})$ is well defined.
\end{remark}

In the next sections we will discuss in more details the previous result in two specific cases.

\section{STFT and super-oscillations: Gaussian window}
\setcounter{equation}{0}
In this section we study the case where the window function $g$ is the Gaussian function given by $\varphi(t):=e^{- \frac{t^2}{2}}$. Thus the signal introduced in \eqref{signalg}, is given by the following expression:
\begin{equation}
\label{signal1}
\displaystyle S_n^{\varphi,x}(t,a)=e^{-\frac{(t-x)^2}{2}} F_n(t,a), \quad \forall t\in \mathbb{R}.
\end{equation}

For the Gaussian window $\varphi$ we can prove the following result: 
\begin{theorem}
\label{norm}	
Let $a>1$ and $x\in \mathbb{R}$. Then
\begin{equation}
\label{norm8}
\displaystyle ||V_{\varphi}(S_n^{\varphi,x})||_{L^2(\mathbb{R}^2)}^{2}= \pi \sum_{j,k=0}^n C_{j}(n,a)C_k(n,a)e^{-\frac{2ix(j-k)}{n}- \frac{(k-j)^2}{n^2}}.
\end{equation}

\end{theorem}
\begin{proof}
According to Lemma \ref{norm4} we need to compute $\mathcal{F}(\psi_{\varphi,x})$ with $\psi_{\varphi,x}(t)=| \varphi(t-x)|^2$. From the definition of the Fourier transform, by making the following change of variables $s=t-x$ and using the Gaussian integral, see \eqref{gint}, we get
\begin{eqnarray*}
\nonumber
\int_{- \infty}^{+ \infty} e^{-(t-x)^2+ \frac{2 i t}{n}(k-j)} dt&=& \int_{- \infty}^{+ \infty} e^{-s^2+ \frac{2 i (s+x)}{n}(k-j)} ds\\
\nonumber
&=& e^{ \frac{2i x(k-j)}{n}}\int_{- \infty}^{+ \infty} e^{- s^2+ \frac{2 i s (k-j)}{n}} ds\\
\label{norm6}
&=&  \sqrt{\pi} e^{\frac{2ix(k-j)}{n}- \frac{(k-j)^2}{n^2}}.
\end{eqnarray*}
Finally, by using the fact that $\| \varphi \|_{L^2(\mathbb{R})}^2=\sqrt{\pi}$ we get the final result.
\end{proof}
Now, we focus on the double summation that appears in formula \eqref{norm8}.

\begin{lemma}
\label{dan}
Let $a>1$ and $n \in \mathbb{N}$, then 
$$\displaystyle \sum_{j,k=0}^n C_j(n,a)C_k(n,a) e^{\frac{2ix}{n}(k-j)} e^{- \frac{(k-j)^2}{n^2}}=\frac{1}{\sqrt{\pi}}\int_{\mathbb{R}}|\phi_{n,a}(s)|^2e^{-s^2}ds:= \frac{1}{\sqrt{\pi}}||\phi_{n,a}||^2_{L^2(\mathbb{R},e^{-s^2}ds)},$$
where we have set \begin{equation}
 \phi_{n,a}(s)=\sum_{\ell=0}^{n}C_\ell(n,a)e^{-\frac{2 \ell^2}{n^2}+\frac{2 \ell}{n}(s-ix)}.
\end{equation}
\end{lemma}
\begin{proof}
We start by making easy manipulations inside the double summation $j$, $k$, and we get
\begin{eqnarray*}
 \sum_{j,k=0}^n C_j(n,a)C_k(n,a) e^{\frac{2ix}{n}(k-j)} e^{- \frac{(k-j)^2}{n^2}}&=& \sum_{j,k=0}^n C_j(n,a)C_{k}(n,a) e^{\frac{2ix}{n}(k-j)} e^{- \frac{k^2}{n^2}} e^{- \frac{j^2}{n^2}} e^{2\frac{kj}{n^2}}\\
 &=& \sum_{j,k=0}^n  \left(C_j(n,a) e^{- \frac{j^2}{n^2}} e^{- \frac{2i jx}{n}}\right) \left( C_k(n,a) e^{- \frac{k^2}{n^2}+ \frac{2i kx}{n}}\right) e^{\frac{2 kj}{n^2}}.
\end{eqnarray*}
By formula \eqref{trick} we get
\begin{eqnarray}
\nonumber
&&\sum_{j,k=0}^n  \left(C_j(n,a) e^{- \frac{j^2}{n^2}} e^{- \frac{2i jx}{n}}\right) \left( C_k(n,a) e^{- \frac{k^2}{n^2}+ \frac{2i kx}{n}}\right) e^{\frac{2kj}{n^2}}\\
\label{kernel}
&=& \pi \sum_{j,k=0}^n  \left(C_j(n,a) e^{- \frac{j^2}{n^2}} e^{- \frac{2i jx}{n}}\right) \left( C_k(n,a) e^{- \frac{k^2}{n^2}+ \frac{2i kx}{n}}\right)
\langle A_{\frac{\sqrt{2}k}{n}}, A_{\frac{ \sqrt{2}j}{n}} \rangle_{L^2(\mathbb{R})},
\end{eqnarray}
where $A_{\frac{\sqrt{2}k}{n}}$, $A_{\frac{ \sqrt{2}j}{n}}$ denote the Bargmann kernel.
Now, by using the expression of the kernel of the Bargmann transform, see \eqref{kernelb}, we get
\begin{eqnarray*}
&&\pi\sum_{j,k=0}^n  \left(C_j(n,a) e^{- \frac{j^2}{n^2}} e^{- \frac{2i jx}{n}}\right) \left( C_k(n,a) e^{- \frac{k^2}{n^2}+ \frac{2i k x}{n}}\right)\langle A_{\frac{ \sqrt{2}k}{n}}, A_{\frac{\sqrt{2}j}{n}} \rangle_{L^2(\mathbb{R})}\\
&=& \frac{1}{\sqrt{\pi}} \sum_{j,k=0}^n \left(C_j(n,a) e^{- \frac{j^2}{n^2}} e^{- \frac{2i jx}{n}}\right) \left( C_k(n,a) e^{- \frac{k^2}{n^2}+ \frac{2i k x}{n}}\right) \int_{\mathbb{R}} e^{- \frac{1}{2} \left( \frac{2k^2}{n^2}+s^2\right)+ \frac{2ks}{n}} e^{- \frac{1}{2} \left( \frac{2j^2}{n^2}+s^2\right)+ \frac{2 js}{n}} ds\\
&=&  \frac{1}{\sqrt{\pi}}\int_{\mathbb{R}} e^{-s^2} \left( \sum_{j=0}^n C_j(n,a) e^{- \frac{2j^2}{n^2}+j \left( \frac{2s}{n}-\frac{2ix}{n}\right)} \right)\left( \sum_{k=0}^n C_k(n,a) e^{- \frac{2k^2}{n^2}+k\left( \frac{2s}{n}+\frac{2ix}{n}\right)} \right) ds\\
&=& \frac{1}{\sqrt{\pi}}\int_{\mathbb{R}} e^{-s^2} \overline{\phi_{n,a}(s)} \phi_{n,a}(s) ds\\
&=& \frac{1}{\sqrt{\pi}} \|\phi_{n,a}\|^2_{L^2(\mathbb{R},e^{-s^2}ds)}.
\end{eqnarray*}
This proves the result.
\end{proof}

By combining Theorem \ref{norm} and Lemma \ref{dan} we have the following result
\begin{corollary}
Let $a>1$ and $x\in \mathbb{R}$. Then
	$$ ||V_{\varphi}(S_n^{\varphi,x})||_{L^2(\mathbb{R}^2)}= \sqrt{\pi} ||\phi_{n,a}||_{L^2(\mathbb{R},e^{-s^2}ds)}.$$
\end{corollary} 
Now, we write an explicit expression of the STFT applied to the signal \eqref{signal1}.
\begin{lemma}
\label{gauss}
Let $a>1$. Let $x\in\mathbb{R}$ be fixed. Then, for every $u,\eta \in\mathbb{R}$ we have 

\begin{equation}
		\label{a3}
V_{\varphi}(S_n^{\varphi,x}) (u,\eta)=\sqrt{\pi} e^{- \frac{\eta^2}{4} - \frac{i(u+x) \eta }{2}- \frac{(u-x)^2}{4}} \sum_{j=0}^n C_{j}(n,a) e^{\frac{ i (u+x) \omega_j}{2}} e^{- \frac{\omega_j^2}{4} + \frac{\eta \omega_j}{2}},
\end{equation}
where $\omega_j=1- \frac{2j}{n}.$
\end{lemma}
\begin{proof}
By Theorem \ref{main} and the form of the reproducing kernel of the Gabor space in the case the window function is a Gaussian function, see \eqref{Gaborg}, we have
\begin{eqnarray*}
	V_g (S_n^{\varphi,x})(u, \eta)&=& \sum_{j=0}^n C_j(n,a) \langle M_{\omega_j}T_x \varphi, M_{\eta}T_x \varphi \rangle\\
	&=& \sqrt{\pi} \sum_{j=0}^n C_j(n,a) e^{\frac{i(u+x)(\omega_j-\eta)}{2}} e^{- \frac{(u-x)^2}{4}- \frac{(\eta- \omega_j)^2}{4}}.
\end{eqnarray*}

\end{proof}

In \cite{A} the authors show how the Fock space $\mathcal{F}(\mathbb{C})$ can be identified with the Gabor space $\mathcal{G}_\varphi$ with a Gaussian window using the operators $E$ given by
\begin{eqnarray}
\nonumber
E&:& \mathcal{G}_{\varphi}(\mathbb{C}) \to \mathcal{F}(\mathbb{C})\\
\label{ope}
&&f \to \mathcal{M}_zf,
\end{eqnarray}
where $\mathcal{M}_z=e^{\frac{|z|^2}{2}-\frac{i x \omega}{2}}$, $z=\frac{x+i \omega}{\sqrt{2}}$. The inverse of $\mathcal{M}_z$ is given by $\mathcal{M}_z^{-1}=e^{-\frac{|z|^2}{2}+\frac{i x \omega}{2}}$.
\\In the following result we provide a relation between the STFT applied to the signal $S_n^{\varphi,x}$, and the normalized Fock kernel through the operator $M_z$.
\begin{theorem}
\label{table}
Let $a>1$. Let $x\in\mathbb{R}$ be fixed. Then, for every $u,\eta \in\mathbb{R}$ we have 
\begin{equation}
\label{opeM}
V_{\varphi}(S_n^{\varphi,x}) (u,\eta)=\pi e^{ux}\mathcal{M}_{p}^{-1}\left(\sum_{j=0}^{n}C_j(n,a)k_{\frac{w_j}{2}}(\bar{p})\right),
\end{equation}
where $k_{\frac{w_j}{2}}$ is the normalized reproducing kernel of the Fock space, $\omega_j=1- \frac{2j}{n}$, and $p=\eta-i(u+x)$.
\end{theorem}
\begin{proof}
We start by making easy  manipulations of \eqref{a3}, we obtain
\begin{eqnarray*}
	\nonumber
	V_g (S_n^{\varphi,x})(u, \eta)&=& \sqrt{\pi} e^{- \frac{i  (u+x) \eta}{2}} e^{- \frac{(u-x)^2}{4}} \sum_{j=0}^n C_j(n,a) e^{\frac{ i (u+x)\omega_j}{2}} e^{- \frac{(\eta- \omega_j)^2}{4}}\\
	&=& \sqrt{\pi} e^{- \frac{\eta^2}{4} - \frac{i(u+x) \eta }{2}- \frac{(u-x)^2}{4}} \sum_{j=0}^n C_{j}(n,a) e^{\frac{ i (u+x) \omega_j}{2}} e^{- \frac{\omega_j^2}{4} + \frac{\eta \omega_j}{2}}.
\end{eqnarray*}
Now, we set $p:=\eta-i(u+x)$ . We observe that $M_p^{-1}=e^{- \frac{\eta^2}{4}- \frac{(u+x)^2}{4}-\frac{i(u+x)\eta}{2}}$, hence we get
\begin{eqnarray*}
V_g (S_n^{\varphi,x})(u, \eta)&=&\sqrt{\pi} e^{- \frac{\eta^2}{4} - \frac{i(u+x) \eta }{2}- \frac{(u-x)^2}{4}} \sum_{j=0}^n C_{j}(n,a) e^{ \frac{\omega_j}{2} \overline{p}-\frac{\omega_j^2}{4}}\\
&=& \pi e^{- \frac{(u-x)^2}{4}+ \frac{(u+x)^2}{4}} M_p^{-1} \left( \sum_{j=0}^n C_j(n,a) k_{\frac{\omega_j}{2}}(\bar{p})\right)\\
&=& \pi e^{ ux}\mathcal{M}_{p}^{-1}\left(\sum_{j=0}^{n}C_j(n,a)k_{\frac{\omega_j}{2}}(\bar{p})\right).
\end{eqnarray*}

\end{proof}

\begin{corollary}
Let $a>1$. For every $x$, $\eta$, $u \in \mathbb{R}$ we have 
$$
\mathcal{M}_{p}\left(	V_{\varphi}(S_n^{\varphi,x}) \right)(u,\eta)= \pi \sum_{j=0}^{n}C_j(n,a)k_{\frac{\omega_j}{2}}(\bar{p}), 
$$
where $\omega_j=1- \frac{2j}{n}$ and $p= \eta-iu$.
\end{corollary}
\begin{proof}
To show the result it is enough to take $x=0$ in \eqref{opeM}.
\end{proof}

\begin{proposition}
Let $a>1$. For every $x$, $\eta$, $u \in \mathbb{R}$ we have
$$ \| \mathcal{M}_{p} V_g (S_n^{\varphi,x})\|_{\mathcal{F}(\mathbb{C})}^2= \pi \sum_{j,k=0}^n C_{j}(n,a)C_k(n,a)e^{- \frac{(k-j)^2}{n^2}},$$
where $p=\eta-iu$.
\end{proposition}
\begin{proof}
Since the operator $M_p$ is an unitary isomorphism from the Gabor space into the Fock space, see \cite[Prop. 4]{A}, we get
$$ \| \mathcal{M}_{p} V_g (S_n^{\varphi,x})\|_{\mathcal{F}(\mathbb{C})}^2= \|V_g (S_n^{\varphi,x})\|_{L^2(\mathbb{R})}^2.$$
By Theorem \ref{norm}, with $x=0$, we obtain the result.
\end{proof}

\section{STFT and superoscillations: Hermite window}
\setcounter{equation}{0}
In this section we consider the case where the window function $g$ of the signal \eqref{sig} is the Hermite function, i.e. $g(t)=h_{m}(t)$. Precisely, we will consider the following signal
\begin{equation}
\label{sigeq}
S_{n}^{h_m,x}(t,a)= (T_{x}h_{m})(t) F_{n}(t, a), \qquad m \geq0.
\end{equation}
Now, we show an explicit expression of the STFT, with window function $h_m$, applied to the above signal.

\begin{proposition}
\label{hermite1}
Let $a>1$ and $(x,u, \eta) \in \mathbb{R}^3$. For every $j=0,1,...n$ we set $ \omega_j= 1- \frac{2j}{n}$, then we have
\begin{eqnarray*}
V_{h_m}(S_{n}^{h_m,x})(u, \eta)&=& \sqrt{\pi} e^{\frac{xu-i(u+x)\eta}{2}}e^{- \frac{ (x^2+u^2+ \eta^2)}{4}} \sum_{j=0}^n C_j(n,a) e^{\frac{1}{2}[ i\omega_j(u+x)+\omega_j \eta- \frac{ \omega_j^2}{2}]}\times
\\ && \times L_m^0 \left( \frac{(x-u)^2+(\omega_j- \eta)^2}{2}\right),
\end{eqnarray*}
where $L_n^0$ are the Laguerre polynomials defined in \eqref{laguerre}.
\end{proposition}
\begin{proof}
The result follows by combing Theorem \ref{main} and formula \eqref{Gaborh}.
\end{proof}
\begin{remark}
If we consider $m=0$ in Proposition \ref{hermite1} we reobtain the same result given by Lemma \ref{gauss}.
\end{remark}


Now, we tackle the problem to figure out the behaviour of the short time Fourier transform $V_{h_k}$ applied to the signal \eqref{sigeq} where the Hermite functions taken into consideration have a different order of the one of the window function. Precisely, in the next result we are going to apply the short time $V_{h_k}$ to the signal \eqref{sigeq} with $k \neq m$.

\begin{theorem}
\label{differnt}
	Let $a>1$ and $(x, u, \eta) \in \mathbb{R}^3$. Let us consider $ \omega_j:= 1- \frac{2j}{n}$. For $k$, $m \geq 0$ we have
	\begin{equation}
		V_{h_k}(S_{n}^{h_m,x})(u, \eta)= i^{k+m} \sum_{j=0}^n C_{j}(n,a) \left( M_{x}h_k*M_{u}h_{m}\right)(\omega_j- \eta).
	\end{equation}
\end{theorem}
\begin{proof}
	By the definition of the STFT, see \eqref{SFT}, we get
\begin{eqnarray}
\nonumber
V_{h_k}(S_{n}^{h_m,x})(u, \eta)&=& \int_{\mathbb{R}} (T_{x} h_{k})(t) F_{n}(t,a) e^{-i \eta t} h_{m}(t-u) dt\\
\nonumber
&=& \sum_{j=0}^n C_{j}(n,a) \int_{\mathbb{R}} e^{- i \eta t} e^{i \left(1- \frac{2j}{n}\right)t} h_{k}(t-x)h_{m}(t-u) dt\\
\nonumber
&=& \sum_{j=0}^n C_{j}(n,a) \int_{\mathbb{R}} e^{- it \left( \eta- \omega_j\right)} h_{k}(t-x)h_{m}(t-u) dt\\
\label{herm0}
&=& \sum_{j=0}^n C_{j}(n,a) \mathcal{F} \left(T_x h_k T_u h_m\right)(\eta- \omega_j).
\end{eqnarray}	
Since $ \mathcal{F}(h_k)(\lambda)= (-i)^k h_k(\lambda)$ we get
$$( T_{x} h_k)(\lambda)=  i^k (T_{x} \mathcal{F}(h_k))(\lambda)= i^k  \mathcal{F}(M_xh_k)(\lambda).$$	
This implies that
$$ \left(T_x h_k T_u h_m\right)(\lambda)= i^{k+m} \mathcal{F}(M_x h_k)(\lambda) \mathcal{F}(M_u h_m)(\lambda).$$
Now, by using the properties of the Fourier transform with respect to the convolution and involution we get
\begin{align}
\nonumber
\mathcal{F}(T_x h_k T_u h_m)(\eta -\omega_j)&= i^{k+m} \mathcal{F} \left( \mathcal{F}(M_x h_x) \mathcal{F}(M_u h_m)\right)(\eta -\omega_j)\\
\nonumber
&= i^{k+m}  \mathcal{F} \left( \mathcal{F}\left( M_x h_k*M_u h_m\right)\right)(\eta -\omega_j)\\
\label{herm1}
&= i^{k+m} \left(M_x h_k * M_u h_m\right)(\omega_j- \eta).
\end{align}	
By plugging \eqref{herm1} into \eqref{herm0} we get
$$ V_{h_k}(S_{n}^{h_m,x})(u, \eta)= i^{k+m} \sum_{j=0}^n C_{j}(n,a) \left( M_{x}h_k*M_{u}h_{m}\right)(\omega_j- \eta).$$
\end{proof}

\subsection{Technical tools on 2D-Hermite polynomials}

In this subsection we will prove a relation connecting the convolution product of two Hermite functions to the so-called 2D-complex Hermite polynomials, that are defined by
\begin{equation}
\label{compexhermite}
H_{k , \ell}(z,w)= \sum_{j=0}^{min(k, \ell)} (-1)^j j! \binom{k}{j} \binom{\ell}{j} z^{\ell-j} w^{k-j}, \qquad (z,w) \in \mathbb{C}^2.
\end{equation}
These polynomials were first introduced in \cite{I} and studied in detail in \cite{GHS2019} in the case of two complex variables (see also \cite{IZ} for q-analogues of these polynomials). 
We need an auxiliary result.

\begin{lemma}
\label{inte2}
Let $k, m \geq 0$ and $x$, $u$, $\lambda \in \mathbb{R}$. Then we have
\begin{equation}
\label{important}
\int_{- \infty}^{+ \infty} e^{ i t \lambda} h_k(t-x)h_m(t-u) dt=\sqrt{\pi}e^{- \frac{\lambda^2}{4}+\frac{i\lambda (x+u)}{2}}   e^{-\frac{(x-u)^2}{4}} \mathcal{I}_{k,m}(x,u, \lambda),
\end{equation}
where
\begin{equation}
\label{series}
\mathcal{I}_{k,m}(x,u,\lambda):= \sum_{\ell=0}^m 2^{\frac{k+ \ell}{2}} i^{k+ \ell} \binom{m}{\ell} \left(2(x-u)\right)^{m- \ell} 
H_{k, \ell}\left( \frac{ \lambda+i(x-u)}{\sqrt{2}},  \frac{ \lambda+i(x-u)}{\sqrt{2}}\right).
\end{equation}

\end{lemma}
\begin{proof}
By performing the change of variables $s=t-x$ we have to compute the integral
$$  \int_{- \infty}^{+ \infty} e^{i \lambda(s+x)} h_k(s) h_m(s+x-u) ds.$$
Now, we express the Hermite functions in terms of Hermite polynomials, i.e. $ h_k(s)= e^{- \frac{s^2}{2}} H_k(s)$, and we obtain

\begin{eqnarray}
	\nonumber
	&&\int_{- \infty}^{+ \infty} e^{ i \lambda(s+x)} h_k(s) h_m(s+x-u) ds\\
	\nonumber
	&&=e^{ i \lambda x} \int_{- \infty}^\infty e^{i \lambda s- \frac{s^2}{2}- \frac{(s+x-u)^2}{2}}H_k(s)H_m(s+x-u) ds\\
	\nonumber
	&&=  e^{ i \lambda x- \frac{(x-u)^2}{2}} \int_{- \infty}^{+ \infty} e^{ i \lambda s-s^2-s(x-u)} H_k(s)H_m(s+x-u) ds\\
	\label{hermite}
	&&=e^{ i \lambda x- \frac{(x-u)^2}{2}} \int_{- \infty}^{+ \infty} e^{-s^2+ s\left( i \lambda-(x-u)\right)} H_k(s)H_m(s+x-u) ds.
\end{eqnarray}

By using the well-known additive formula for Hermite polynomials, see \cite{AS}, given by
$$ H_{m}(x+y)= \sum_{\ell=0}^m \binom{m}{\ell} H_{\ell}(x) \left(2y\right)^{m- \ell}$$
we get
$$ H_m(s+x-u) = \sum_{\ell=0}^m \binom{m}{\ell} H_{\ell}(s)\left(2(x-u)\right)^{m- \ell}.$$
This together with \eqref{hermite} imply that
\begin{eqnarray}
	\nonumber
&&e^{ i \lambda x- \frac{(x-u)^2}{2}} \int_{- \infty}^{+ \infty} e^{-s^2+( i \lambda-(x-u))s} H_k(s) \left[ \sum_{\ell=0}^m \binom{m}{\ell} H_{\ell}(s) \left(2(x-u)\right)^{m- \ell}\right] ds\\
	\nonumber
	&&= e^{i \lambda x- \frac{(x-u)^2}{2}} \times\\
	\label{hermite2}
	&& \, \, \, \,\, \, \, \,\, \, \times \sum_{\ell=0}^m \binom{m}{\ell} \left(2(x-u)\right)^{m- \ell} \int_{- \infty}^{+ \infty}e^{-s^2+s\left( i \lambda-(x-u)\right)} H_{k}(s)H_{\ell}(s) ds.
\end{eqnarray}

Now, we focus on finding a solution of the integral in \eqref{hermite2}. By using \cite[formula 2.0.1]{F} we can express the product of two Hermite polynomials in terms of a linear combination of Hermite polynomials, i.e.
$$ H_{m}(x)H_{n}(x)= \sum_{j=0}^{min(m,n)} j! 2^{j} \binom{m}{j} \binom{n}{j} H_{m+n-2j}(x),$$
we get
\begin{eqnarray}
	\nonumber
	\int_{- \infty}^{+ \infty}e^{-s^2+s\left(i \lambda-(x-u)\right)} H_{k}(s)H_{\ell}(s) ds &=& \int_{- \infty}^{+ \infty} e^{-s^2+s \left(i \lambda-(x-u)\right)}\left[ \sum_{j=0}^{min(k,l)} j! 2^{j} \binom{k}{j} \binom{\ell}{j} H_{k+ \ell -2j}(s) \right]\\
	\label{hermite3}
	&=& \sum_{j=0}^{min(k, \ell)} j! 2^{j} \binom{k}{j} \binom{\ell}{j} \int_{- \infty}^{+ \infty}e^{- \frac{s^2}{2}+ \sqrt{2} s \frac{( i \lambda -(x-u))}{\sqrt{2}}} h_{k+ \ell-2j}(s) ds.
\end{eqnarray}
Now, we focus on computing the integral in \eqref{hermite3}. We set $z:=\frac{ i \lambda -(x-u)}{\sqrt{2}}$ and by using the action of Segal-Bargmann transform on Hermite functions, see \eqref{action}, we obtain
\begin{eqnarray}
	\nonumber
	\int_{- \infty}^{+ \infty}e^{- \frac{s^2}{2}+ \sqrt{2} s z} h_{k+ \ell-2j}(s) ds&=& e^{\frac{z^2}{2}}\int_{- \infty}^{+ \infty}e^{- \frac{1}{2}(s^2+z^2)+\sqrt{2} s z} h_{k+ \ell-2j}(s) ds\\
	\nonumber
	&=& \pi^{\frac{3}{4}}e^{\frac{z^2}{2}} \mathcal{B}(h_{k+ \ell-2j})(z)\\
	\label{hermite4}
	&=& \sqrt{\pi} 2^{\frac{k+ \ell-2j}{2}} e^{\frac{z^2}{2}}z^{k+ \ell -2j}.
\end{eqnarray}
By plugging \eqref{hermite4} in \eqref{hermite3} we get
\begin{equation}
	\label{hermite6}
 \int_{- \infty}^{+ \infty}e^{-s^2+s\left( i \lambda-(x-u)\right)} H_{k}(s)H_{\ell}(s) ds= \sqrt{\pi} 2^{\frac{k+ \ell}{2}} e^{\frac{\left( i \lambda-(x-u)\right)^2}{4}} \mathcal{P}_{k,\ell}\left( \frac{ i \lambda-(x-u)}{\sqrt{2}}\right),
\end{equation}
where $\mathcal{P}_{k,\ell}(z):=\sum_{j=0}^{min(k, \ell)} j! \binom{k}{j} \binom{\ell}{j} z^{k+ \ell-2j}$. 
\\Now, we show that these polynomials are related to the 2D-Hermite polynomials introduced in \eqref{compexhermite}. Indeed we can write
$ z=i \left(\frac{ \lambda+i(x-u)}{\sqrt{2}}\right).$
Hence we get
\begin{eqnarray}
	\nonumber
	\mathcal{P}_{k, \ell}(z)&=& \sum_{j=0}^{min(k, \ell)} j! \binom{k}{j} \binom{\ell}{j} i^{k+ \ell-2j} \left(\frac{ \lambda+i(x-u)}{\sqrt{2}}\right)^{k+ \ell-2j}\\
	\nonumber
	&=& i^{k+ \ell}\sum_{j=0}^{min(k, \ell)} (-1)^j j! \binom{k}{j} \binom{\ell}{j}  \left(\frac{ \lambda+i(x-u)}{\sqrt{2}}\right)^{k+ \ell-2j}\\
	\label{hermite5}
	&=& i^{k+ \ell}  H_{k, \ell}\left( \frac{ \lambda+i(x-u)}{\sqrt{2}},  \frac{ \lambda+i(x-u)}{\sqrt{2}}\right).
\end{eqnarray}
By plugging \eqref{hermite5} into \eqref{hermite6} we get 
\begin{eqnarray}
	\nonumber
	\int_{- \infty}^{+ \infty}e^{-s^2+s\left( i \lambda-(x-u)\right)} H_{k}(s)H_{\ell}(s) ds&=& \sqrt{\pi} 2^{\frac{k+ \ell}{2}} e^{\frac{\left(i \lambda-(x-u)\right)^2}{4}} i^{k+ \ell} \times\\
	\label{hermite7}
	&&\times H_{k, \ell}\left( \frac{ \lambda+i(x-u)}{\sqrt{2}},  \frac{ \lambda+i(x-u)}{\sqrt{2}}\right).
\end{eqnarray}
Finally, we substitute \eqref{hermite7} into \eqref{hermite2} and we get the final result.
\end{proof}

\begin{remark}
In the next subsection we will compute a closed formula of the expression $\mathcal{I}_{k,m}(x,u,\lambda)$ using a careful adaptation of the properties of the Weyl-Heisenberg symmetry and the Bargmann-Fock representation.
\end{remark}

An easy consequence of the previous result is the following corollary.
\begin{corollary}
Let $m \geq 0$ and $x$, $ \lambda \in \mathbb{R}$. Then we have
\begin{equation}
	\label{important1}
	\int_{- \infty}^{+ \infty} e^{ i t \lambda} \left(h_m(t-x)\right)^2dt=\sqrt{\pi} (-2)^m e^{- \frac{ \lambda^2}{4}+i\lambda x} H_{m,m}\left(\frac{\lambda}{\sqrt{2}}, \frac{\lambda}{\sqrt{2}} \right) .
\end{equation}
\end{corollary}
\begin{proof}
The result follows by taking $k=m$ and $x=u$ in the integral \eqref{important}.
\end{proof}

Based on the previous result we can prove the following.
\begin{proposition}
\label{mstarherm}
Let $k$, $m \geq 0$ and $(x, u) \in \mathbb{R}^2$. Then we have
\begin{equation}
\label{final1}
(M_{x}h_k*M_{u}h_{m})(\lambda)= (-i)^{m+k} \sqrt{\pi} e^{- \frac{\lambda^2}{4}+\frac{i \lambda (x+u)}{2}}   e^{-\frac{(x-u)^2}{4}} \mathcal{I}_{k,m}(x,u, \lambda),
\end{equation}
where the terms $\mathcal{I}_{k,m}(x,u, \lambda)$ is defined in \eqref{series}.
\end{proposition}
\begin{proof}
By the well-known properties of Fourier transform on modulation and translation and the fact that $ \mathcal{F}(h_k)(\lambda)=(-1)^k h_k(\lambda)$ we get
\begin{eqnarray*}
\left(M_x h_k * M_u h_m\right)(\lambda)&=& (-i)^{k+m} \left(\mathcal{F}^{-1}(T_x h_k)*\mathcal{F}^{-1}(T_u h_m) \right)(\lambda)\\
&=& (-i)^{k+m} \mathcal{F}^{-1}[T_x h_k T_u h_m](\lambda)\\
&=& (-i)^{k+m} \int_{-\infty}^\infty e^{ i t \lambda} h_{k}(t-x)h_{m}(t-u) dt.
\end{eqnarray*}
Finally the result follows by Lemma \ref{inte2}.
\end{proof}
From this result we can generate several interesting consequences.

\begin{corollary}
\label{cor1}
For $k$, $m \geq 0$ the convolution product of two Hermite functions is given by
\begin{equation}
\label{conv}
(h_k*h_m)(\lambda)= \sqrt{\pi} 2^{\frac{m+k}{2}} e^{- \frac{\lambda^2}{4}} H_{k,m} \left( \frac{\lambda}{\sqrt{2}}, \frac{\lambda}{\sqrt{2}}\right).
\end{equation}
Moreover, if we consider $ x \in \mathbb{R}$, we have
\begin{equation}
\label{Mconv}
(M_x h_k* M_x h_m)(\lambda)= \sqrt{\pi} 2^{\frac{k+m}{2}}  e^{-  \frac{\lambda^2}{4}+ i \lambda x} H_{k, m}\left(\frac{\lambda}{\sqrt{2}},\frac{\lambda}{\sqrt{2}} \right).
\end{equation}
\end{corollary}
\begin{proof}
Formula \eqref{conv} follows by taking $x=u=0$ in \eqref{final1}. Similarly, if we take $x=u$ in \eqref{final1} we get \eqref{Mconv}.
\end{proof}

\begin{remark}
It is possible to write the complex-Hermite polynomials in terms of the Laguerre polynomials (see \cite{II}). Precisely, we can write 
$$H_{k,m}(z, \bar{z})=(-1)^{\min{(k,m)}}\min{(k,m)}! e^{i \theta (k-m)} |z|^{|k-m|} L_{\min(k,m)}^{|k-m|}(|z|^2), \qquad z=|z| e^{i \theta}.$$
This implies that we can write the convolution product of two Hermite functions as
$$ (h_k*h_m)(\lambda)= \sqrt{\pi}  2^{\frac{m+k}{2}} (-1)^{\min{(k,m)}}\min{(k,m)}! \lambda^{|m-k|}e^{- \frac{\lambda^2}{4}} L^{|m-k|}_{\min(k,m)} \left(\frac{\lambda^2}{2}\right).$$
\end{remark}
 
\begin{proof}
This result follows by taking $x=u$ in formula \eqref{final1}.
\end{proof}

From the above results we obtain the following generating formula.

\begin{proposition}
\label{bnew}
For fixed $u$ and $v$ and $x \in \mathbb{R}$ we have the following generating function:
\begin{equation}
\label{gen0}
\sum_{k,m=0}^{\infty} \frac{u^k v^k}{2^{\frac{k+m}{2}}k! m!} \left(M_x h_k* M_x h_m\right)(\lambda)= \sqrt{\pi} e^{- \frac{\lambda^2}{4}+ \lambda \left(ix+ \frac{(u+v)}{\sqrt{2}}\right)} e^{-uv}.
\end{equation}
\end{proposition}
\begin{proof}
In order to get the desired formula we have to use the following generating property of the 2D-Hermite polynomials
\begin{equation}
\label{gen}
\sum_{n,m=0}^\infty H_{m,n}(z,w) \frac{u^m v^m}{m! n!}= e^{uz+vw-uv}; \qquad \forall (z,w) \in \mathbb{C}^2,
\end{equation}
see \cite{IS}.
By Corollary \ref{cor1} and the generating formula \eqref{gen} we have
\begin{eqnarray*}
\sum_{k,m=0}^{\infty} \frac{u^k v^k}{2^{\frac{k+m}{2}}k! m!} \left(M_x h_k* M_x h_m\right)(\lambda)&=& \sqrt{\pi} e^{- \frac{\lambda^2}{4}+ i \lambda x} \sum_{k,m=0}^\infty \frac{u^k v^m}{k! m!} H_{k,m}\left( \frac{\lambda}{\sqrt{2}}, \frac{\lambda}{\sqrt{2}}\right)\\
&=& \sqrt{\pi}  e^{\frac{u\lambda}{\sqrt{2}}+ \frac{v\lambda}{\sqrt{2}}-uv}e^{- \frac{\lambda^2}{4}+ i \lambda x}\\
&=& \sqrt{\pi} e^{- \frac{\lambda^2}{4}+ \lambda \left(ix+ \frac{(u+v)}{\sqrt{2}}\right)} e^{-uv}.
\end{eqnarray*}
This proves the result.
\end{proof}

As a consequence of the previous result we have the following.

\begin{proposition}
\label{refe}
For fixed $u$ and $v$ and $x \in \mathbb{R}$ we have the following:
$$ \sum_{k,m=0}^\infty \frac{u^k v^m}{2^{\frac{k+m}{2}} k! m!} (-i)^{k+m} h_{k}(\lambda-x) h_{m}(\lambda-x)=2 \pi  e^{-uv- (x- \lambda)^2+\frac{(u+v)^2}{2} +\sqrt{2}i(x-\lambda)(u+v)}.$$
\end{proposition}
\begin{proof}
We start by applying the Fourier transform to formula \eqref{gen0} and by the fact that $ \mathcal{F}(M_x h_k)(\lambda)= (-i)^k h_k(\lambda-x)$ we have
\begin{eqnarray}
\nonumber
\mathcal{F} \left(\sum_{k,m=0}^{\infty} \frac{u^k v^k}{2^{\frac{k+m}{2}}k! m!} \left(M_x h_k* M_x h_m\right)(\lambda)\right)&=& \sum_{k,m=0}^{\infty} \frac{u^k v^k}{2^{\frac{k+m}{2}}k! m!}  \mathcal{F}(M_x h_k)(\lambda) \mathcal{F}(M_x h_m)(\lambda)\\
\label{fou1}
&=& \sum_{k,m=0}^{\infty} \frac{u^k v^k}{2^{\frac{k+m}{2}}k! m!}(-i)^{k+m} h_k(\lambda-x) h_m(\lambda-x).
\end{eqnarray}
On the other hand, by Proposition \ref{bnew}, we have to compute the Fourier transform of the function
$$ \psi_{(u,v)}^x (t)= \sqrt{\pi} e^{- \frac{t^2}{4}+ t \left(i x+ \frac{(u+v)}{\sqrt{2}}\right)-uv}.$$
The Gaussian integral, see \eqref{gint}, leads to the following
\begin{eqnarray}
\nonumber
\mathcal{F}(\psi_{(u,v)}^x)(t)&=&\sqrt{\pi} e^{-uv} \int_{\mathbb{R}} e^{- i t \lambda} e^{- \frac{t^2}{4}+ t \left(i x+ \frac{(u+v)}{\sqrt{2}}\right)} dt\\
\nonumber
&=&\sqrt{\pi} e^{-uv} \int_{\mathbb{R}} e^{- \frac{t^2}{4}+t \left[i (x- \lambda)+ \frac{(u+v)}{\sqrt{2}}\right]} dt\\
\nonumber
&=& 2 \pi e^{-uv} e^{\left(i (x- \lambda)+ \frac{(u+v)}{\sqrt{2}}\right)^2}\\
\label{fou2}
&=& 2 \pi  e^{-uv- (x- \lambda)^2+\frac{(u+v)^2}{2} +\sqrt{2}i(x- \lambda)(u+v)}.
\end{eqnarray}
By putting equal \eqref{fou1} and \eqref{fou2} we get the desired result.

\end{proof}

\begin{remark}
We observe that the result proved in Proposition \ref{refe} can be also obtained by using the following well-known identity of the Hermite polynomials:
$$ e^{2xz-z^2}=\sum_{n=0}^{\infty} \frac{H_n(x)}{n!} z^n, \qquad \forall (x,z) \in \mathbb{R} \times \mathbb{C}.$$
\end{remark}

\subsection{An application of the Weyl-Heisenberg symmetry and the Bargmann-Fock representation}
\setcounter{equation}{0}

Inspired from the Weyl-Heisenberg symmetry and the Bargmann-Fock representation we compute an explicit expression of the function $\mathcal{I}_{k,m}(x,u,\lambda)$, defined in \eqref{series}. This allows to improve the result of Lemma \ref{inte2}. \\

For a given complex parameter $c=a+ib\in \mathbb{C}$ we consider the commutative diagram 
$$\xymatrix{
    \mathcal{F}(\mathbb{C}) \ar[r]^{\mathcal{W}_{a+ib}} \ar[d]_{\mathcal{B}^{-1}} & \mathcal{F}(\mathbb{C}) \\ L^2(\mathbb{R}) \ar[r]_{M_bT_a} & L^2(\mathbb{R}) \ar[u]_{\mathcal{B}}
  }$$ leading to the Weyl-type operator $\mathcal{W}_c:\mathcal{F}(\mathbb{C})\longrightarrow \mathcal{F}(\mathbb{C})$ that can be defined as the unitary operator mapping the Fock space $\mathcal{F}(\mathbb{C})$ onto itself. More precisely, the Weyl-type operator can be computed by composing the modulation and translation operators via the Bargmann transform as follows \begin{equation}
\mathcal{W}_{a+ib}:=\mathcal{B}M_bT_a\mathcal{B}^{-1}.
\end{equation}

It is important to note that the idea of using the above operator was discussed by Kehe Zhu in \cite{Z1} to connect the modulation and the translation operators with the Weyl operator on the Fock space. The connection with the Heisenberg group and representation theory is nicely exposed in \cite[Chapter 9]{GK2001} and \cite[Chapter 1]{Folland}.\\ \\

 First, we prove an intermediate result that will allow us to compute $\mathcal{I}_{k,m}(x,u,\lambda)$.
\begin{lemma}\label{NewLem}
Let $a,b\in \mathbb{R}$ and $k, m\geq 0$ fixed. Then, it holds that

\begin{equation}\label{F1}
\sum_{\ell=0}^m 2^{\frac{\ell}{2}} i^{\ell} \binom{m}{\ell} \left(-2a\right)^{m- \ell} 
 H_{k, \ell}\left( \frac{ b-ia}{\sqrt{2}},  \frac{ b-ia}{\sqrt{2}}\right)=\frac{(-\sqrt{2})^m}{i^k}H_{k,m}\left(\frac{a+ib}{\sqrt{2}},\frac{a-ib}{\sqrt{2}}\right).
\end{equation}
\end{lemma}
\begin{proof}

By means of the connection of the Heisenberg and the Bargmann-Fock representations, we will prove this result in 5 main steps. To this end, we compute the action of the operator $\mathcal{W}_{a+ib}$ on the classical orthonormal basis of the Fock space given by $e_m(z):=\frac{z^m}{\sqrt{m! \pi}}$ where $m\geq 0$. We use two different methods of computations (Step 1 and Step 2). Then, we conclude by comparing the results of these two methods.

\begin{enumerate}
\item[i)] \textbf{Step 1:} First, we prove that 

\begin{equation}\label{DF1}
\mathcal{W}_{a+ib}(e_m)(z)=\frac{1}{\sqrt{m! \pi}}e^{-\frac{(a^2+b^2)}{4}+\frac{iab}{2}}e^{\frac{z(a+ib)}{\sqrt{2}}}\left(z-\frac{(a-ib)}{\sqrt{2}}\right)^m,
\end{equation}
for every $z\in\mathbb{C}$ and $m\geq 0$. Indeed, we check this fact by developing direct computations using properties of the classical Segal-Bargmann transform mapping the Hermite functions onto the complex monomials, i.e. $\mathcal{B}(\psi_m)(z)=e_m(z)$. By \eqref{kernelb} we have

\begin{align*}
\displaystyle \mathcal{W}_{a+ib}(e_m)(z)&= \mathcal{B}(M_b T_a\psi_m)(z)\\
&=\frac{1}{\pi^{3/4}}\int_{\mathbb{R}}e^{-\frac{1}{2}(z^2+x^2)+\sqrt{2}zx}e^{ibx}\psi_m(x-a)dx\\
&=e^{-\frac{a^2}{2}-\frac{z^2}{2}+a(\sqrt{2}z+ib)}\frac{1}{\pi^{3/4}}\int_{\mathbb{R}}e^{-\frac{t^2}{2}+\sqrt{2}t\left(z+\frac{ib-a}{\sqrt{2}}\right)}\psi_m(t)dt\\
&=e^{-\frac{a^2}{2}-\frac{z^2}{2}+a\sqrt{2}z+iab+\frac{z^2}{2}+\frac{(ib-a)^2}{4}+\frac{z(ib-a)}{\sqrt{2}}} \mathcal{B}(\psi_m)\left(z+\frac{ib-a}{\sqrt{2}}\right)\\
&=\frac{1}{\sqrt{m!\pi} }e^{-\frac{a^2+b^2}{4}+\frac{iab}{2}}e^{\frac{z(a+ib)}{\sqrt{2}}}\left(z-\left(\frac{a-ib}{\sqrt{2}}\right)\right)^m.
\end{align*}

\item[ii)] \textbf{Step 2:} In this second part we prove that 
\begin{equation}\label{S2}
\displaystyle \mathcal{W}_{a+ib}(e_{m})(z)=\frac{1}{\sqrt{m!2^m \pi}}e^{-\frac{(a^2+b^2)}{4}+\frac{iab}{2}}\sum_{k=0}^{\infty}\frac{z^k}{k!\sqrt{2^k}}\mathcal{I}_{k,m}(0,a,b)
\end{equation}
where we have
$$\mathcal{I}_{k,m}(0,a,b)=2^{\frac{k}{2}}i^k\sum_{\ell=0}^{m}2^{\frac{\ell}{2}}i^\ell {m \choose \ell} (-2a)^{m-\ell} H_{k,\ell}\left(\frac{b-ia}{\sqrt{2}}, \frac{b-ia}{\sqrt{2}}\right).
$$
To prove \eqref{S2} we have to use the well-known fact that we can write the kernel of the Bargmann transform in terms of generating function of Hermite functions:
$$ A(z,x)= \sum_{n=0}^{\infty} \psi_n(x) e_n(z).$$
By using another time that $\mathcal{B}(\psi_m)(z)=e_{m}(z)$ we get
\begin{align*}
 \mathcal{W}_{a+ib}(e_{m})(z)&=\mathcal{B}(M_bT_a\psi_m)(z)\\
&=\int_{\mathbb{R}}A(z,x)M_b\psi_m(x-a)dx\\
&=\sum_{k=0}^{\infty}e_{k}(z)\int_{\mathbb{R}}e^{ibx}\psi_k(x)\psi_m(x-a)dx\\
&=\frac{1}{\pi \sqrt{ m! 2^m}}\sum_{k=0}^{\infty}\frac{z^k}{k!\sqrt{ 2^k}}\int_{\mathbb{R}}e^{ibx}h_k(x)h_m(x-a)dx.
\end{align*}
By Lemma \ref{inte2} we have 
\begin{eqnarray*}
\mathcal{W}_{a+ib}(e_{m})(z)&=&\frac{1}{ \sqrt{\pi m! 2^m}}e^{-\frac{a^2+b^2}{4}+\frac{iab}{2}}\sum_{k=0}^{\infty}\frac{z^k}{k!\sqrt{ 2^k}}\mathcal{I}_{k,m}(0,a,b)\\
&=& \frac{1}{\sqrt{ \pi 2^m m!}}e^{-\frac{a^2+b^2}{4}+\frac{iab}{2}}\sum_{k=0}^{\infty}\frac{z^k}{k!} \left(i^k\sum_{\ell=0}^m 2^{\frac{\ell}{2}} i^{\ell} \binom{m}{\ell} \left(-2a\right)^{m- \ell} 
H_{k, \ell}\left( \frac{ b-ia}{\sqrt{2}},  \frac{ b-ia}{\sqrt{2}}\right)\right).
\end{eqnarray*}

Finally, we conclude that $\mathcal{W}_{a+ib}(e_m)(z)$ can be rewritten as follows

\begin{equation}\label{DF}
\displaystyle \mathcal{W}_{a+ib}(e_m)(z)=\frac{1}{\sqrt{\pi 2^m m! }}e^{-\frac{a^2+b^2}{4}+\frac{iab}{2}}\sum_{k=0}^{\infty}\frac{z^k}{k!} \alpha_{k,m}(a,b),
\end{equation} 
 
 where we have set 
 
 $$\alpha_{k,m}(a,b)=i^k\sum_{\ell=0}^m 2^{\frac{\ell}{2}} i^{\ell} \binom{m}{\ell} \left(-2a\right)^{m- \ell} 
 H_{k, \ell}\left( \frac{ b-ia}{\sqrt{2}},  \frac{ b-ia}{\sqrt{2}}\right), \quad \forall k\geq 0.$$

\item[iii)] \textbf{Step 3:} Identification of i) and ii). We compare the final formulas of $\mathcal{W}_{a+ib}(e_m)(z)$ obtained in Step 1 and Step 2 (see \eqref{DF1} and \eqref{DF}) to get 
 
 $$ \frac{1}{\sqrt{2^m}} \sum_{k=0}^{\infty}\frac{z^k}{k!} \alpha_{k,m}(a,b)=e^{\frac{z(a+ib)}{\sqrt{2}}}\left(z-\frac{(a-ib)}{\sqrt{2}}\right)^m:=\phi_m(z).$$

 Since the function $\phi_m(z)$ is an entire function by its Taylor's series expansion we have
 $$\phi_m(z)=\sum_{k=0}^{\infty}\frac{z^k}{k!}\phi^{(k)}_{m}(0),$$
 for every $z\in \mathbb{C}$. In particular, we identify the coefficients to deduce
\begin{equation}
\alpha_{k,m}(a,b)=\sqrt{2^m}\phi^{(k)}_{m}(0),
\end{equation}
for every $k,m \geq 0$. This means that 
\begin{equation}
\displaystyle \sum_{\ell=0}^m 2^{\frac{\ell}{2}} i^{\ell} \binom{m}{\ell} \left(-2a\right)^{m- \ell} 
 H_{k, \ell}\left( \frac{ b-ia}{\sqrt{2}},  \frac{ b-ia}{\sqrt{2}}\right)=\frac{\sqrt{2^m}}{i^k}\phi^{(k)}_{m}(0), \quad \forall k,m\geq 0.
\end{equation}
Hence, we just need to calculate $\phi^{(k)}_{m}(0)$ to conclude the proof which will be the main point of the next step.

\item[iv)] \textbf{Step 4:} The first objective of this step is to compute 
$$\phi^{(k)}_{m}(z)=\frac{d^k}{d z^k}\left(e^{\frac{z(a+ib)}{\sqrt{2}}}\left(z-\frac{(a-ib)}{\sqrt{2}}\right)^m\right), \quad \forall k, m\geq 0.$$Then, we will evaluate in $z=0$ to calculate $\phi^{(k)}_{m}(0)$. Let $k,m\geq 0$, we treat two cases when $k\leq m$ and when $k>m$.

\begin{itemize}
\item[•] \textit{Case 1:} We assume  $k\leq m$. Then, applying the Leibniz formula we get 

\begin{align*}
\displaystyle \phi^{(k)}_{m}(z)&=\sum_{s=0}^{k}{k \choose s} \frac{d^{k-s}}{dz^{k-s}} \left(e^{\frac{z(a+ib)}{\sqrt{2}}}\right) \frac{d^s}{dz^s}\left(\left(z-\frac{(a-ib)}{\sqrt{2}}\right)^m\right)\\
&=\sum_{s=0}^{k}{k \choose s}\left(\frac{a+ib}{\sqrt{2}}\right)^{k-s}e^{\frac{z(a+ib)}{\sqrt{2}}} \frac{m!}{(m-s)!}\left(z-\frac{(a-ib)}{\sqrt{2}}\right)^{m-s}\\
&=e^{\frac{z(a+ib)}{\sqrt{2}}} \left(\sum_{s=0}^{k}s!{k \choose s}{m \choose s}\left(\frac{a+ib}{\sqrt{2}}\right)^{k-s}\left(z-\frac{(a-ib)}{\sqrt{2}}\right)^{m-s}\right).
\end{align*}
Therefore, by taking $z=0$ in the above expression of $\phi^{(k)}_{m}(z)$ and comparing with the definition of 2D-complex Hermite polynomials introduced by Itô in \cite{I} we obtain 

\begin{align*}
\displaystyle \phi^{(k)}_{m}(0)&=(-1)^m\sum_{s=0}^{k}(-1)^ss!{k \choose s}{m \choose s}\left(\frac{a+ib}{\sqrt{2}}\right)^{k-s}\left(\frac{a-ib}{\sqrt{2}}\right)^{m-s}\\
&=(-1)^mH_{k,m}\left(\frac{a+ib}{\sqrt{2}}, \frac{a-ib}{\sqrt{2}} \right).
\end{align*}

\item[•] \textit{Case 2:} We assume $k>m$. In this case, we observe that 
$$ \left(\left(z-\frac{(a-ib)}{\sqrt{2}}\right)^m\right)^{(s)}=0, \quad \forall s\geq m+1.$$

Therefore, using similar computations as in \textit{case 1} we get 
$$\displaystyle \phi^{(k)}_{m}(z)=e^{\frac{z(a+ib)}{\sqrt{2}}} \left(\sum_{s=0}^{m}s!{k \choose s}{n \choose s}\left(\frac{a+ib}{\sqrt{2}}\right)^{k-s}\left(z-\frac{(a-ib)}{\sqrt{2}}\right)^{m-s}\right).$$
So, we have \begin{align*}
\displaystyle \phi^{(k)}_{m}(0)&=(-1)^m\sum_{s=0}^{m}(-1)^ss!{k \choose s}{m \choose s}\left(\frac{a+ib}{\sqrt{2}}\right)^{k-s}\left(\frac{a-ib}{\sqrt{2}}\right)^{m-s}\\
&=(-1)^mH_{k,m}\left(\frac{a+ib}{\sqrt{2}}, \frac{a-ib}{\sqrt{2}} \right)\\
\end{align*}
\end{itemize}

In both cases $k\leq m$ and $k>m$ we got
$$\displaystyle \phi^{(k)}_{m}(0)=(-1)^mH_{k,m}\left(\frac{a+ib}{\sqrt{2}}, \frac{a-ib}{\sqrt{2}} \right).$$

Hence, we conclude

\begin{equation}\label{FF}
\alpha_{k,m}(a,b)=\sqrt{2^m}\phi^{(k)}_{m}(0)=(-1)^m\sqrt{2^m}H_{k,m}\left(\frac{a+ib}{\sqrt{2}}, \frac{a-ib}{\sqrt{2}} \right).
\end{equation}

\item[v)] \textbf{Step 5:} Finally, thanks to the previous steps and in particular using the expression of $\alpha_{k,m}(a,b)$ and the formula \eqref{FF} we conclude that

$$ \sum_{\ell=0}^m 2^{\frac{\ell}{2}} i^{\ell} \binom{m}{\ell} \left(-2a\right)^{m- \ell} 
 H_{k, \ell}\left( \frac{ b-ia}{\sqrt{2}},  \frac{ b-ia}{\sqrt{2}}\right)=\frac{(-\sqrt{2})^m}{i^k}H_{k,m}\left(\frac{a+ib}{\sqrt{2}},\frac{a-ib}{\sqrt{2}}\right).$$

\end{enumerate}

This proves the result.
\end{proof}
Now, we have all the tools to write a compact expression of $ \mathcal{I}_{k,m}(x,u,\lambda)$.

\begin{theorem}\label{NewResult}
Let $u,\lambda, x\in \mathbb{R}$ and $k,m\geq 0$. Then, it holds that
\begin{equation}
\mathcal{I}_{k,m}(x,u,\lambda)=(-1)^m2^{\frac{k+m}{2}}H_{k,m}\left(\frac{u-x+i\lambda}{\sqrt{2}},\frac{u-x-i\lambda}{\sqrt{2}}\right).
\end{equation}
\end{theorem}

\begin{proof}

By setting $w=u-x$ in $\mathcal{I}_{k,m}(x,u,\lambda)$ and by Lemma \ref{NewLem} (with $a=u$ and $b=\lambda$)  we get
\begin{align*}
\mathcal{I}_{k,m}(x,u,\lambda)&=\sum_{\ell=0}^m 2^{\frac{k+ \ell}{2}} i^{k+ \ell} \binom{m}{\ell} \left(2(x-u)\right)^{m- \ell} 
 H_{k, \ell}\left( \frac{ \lambda+i(x-u)}{\sqrt{2}},  \frac{ \lambda+i(x-u)}{\sqrt{2}}\right)\\
&=\sum_{\ell=0}^m 2^{\frac{k+ \ell}{2}} i^{k+ \ell} \binom{m}{\ell} \left(-2w\right)^{m- \ell} 
 H_{k, \ell}\left( \frac{ \lambda-iw}{\sqrt{2}},  \frac{ \lambda-iw}{\sqrt{2}}\right)\\
&=(-1)^m2^{\frac{k+m}{2}}H_{k,m}\left(\frac{w+i\lambda}{\sqrt{2}},\frac{w-i\lambda}{\sqrt{2}}\right)\\
&=(-1)^m2^{\frac{k+m}{2}}H_{k,m}\left(\frac{u-x+i\lambda}{\sqrt{2}},\frac{u-x-i\lambda}{\sqrt{2}}\right).
\end{align*}

\end{proof}

\begin{corollary}
\label{new}
Let $k$, $m \geq 0$ and $x, u,\lambda \in \mathbb{R}$. Then we have

\begin{equation}
\label{star1}
(M_{x}h_k*M_{u}h_{m})(\lambda)= \sqrt{\pi}i^{m-k} 2^{\frac{k+m}{2}} e^{- \frac{\lambda^2}{4}+\frac{i \lambda (x+u)}{2}}   e^{-\frac{(x-u)^2}{4}} H_{k,m}\left(\frac{u-x+i\lambda}{\sqrt{2}},\frac{u-x-i\lambda}{\sqrt{2}}\right).
\end{equation}
\end{corollary}
\begin{proof}
It follows from direct application of Theorem \ref{NewResult} and Proposition \ref{mstarherm}.
\end{proof}
\begin{remark}
From the definition of the 2D-Hermite polynomials it clear that\\ $H_{k,m}\left( \frac{i \lambda}{\sqrt{2}},- \frac{i \lambda}{\sqrt{2}}\right)= (-1)^{k-m} H_{k,m}\left( \frac{\lambda}{\sqrt{2}},\frac{\lambda}{\sqrt{2}}\right)$. Thus by considering $u=x=0$ in \eqref{star1} we get back to the formula in \eqref{conv}.
\end{remark}

\subsection{Applications to superoscillations}
Now, we show how the technical results proved so far are useful to find a relation between the STFT $V_{h_k}$ applied to the signal \eqref{sigeq} and the 2D-complex Hermite polynomials.

\begin{theorem}
\label{important3}
	Let $a>1$ and $(x, u, \eta) \in \mathbb{R}^3$. Let us consider $ \omega_j:= 1- \frac{2j}{n}$. For $k$, $m \geq 0$ we have

\begin{equation}
\label{final}
		V_{h_k}(S_{n}^{h_m,x})(u, \eta)= \sqrt{\pi}   (-1)^m2^{\frac{k+m}{2}}    e^{-\frac{i(x+u)\eta}{2}- \frac{(x-u)^2}{4}}H_{k,m}(\alpha, \bar{\alpha}) \sum_{j=0}^n C_j(n,a) e^{-\frac{ (\omega_j-\eta)^2}{4}+ \frac{i(u+x)\omega_j}{2}}.
\end{equation}

with $\alpha:=\frac{ (u-x)+i(\omega_j-\eta)}{\sqrt{2}}$.

\end{theorem}
\begin{proof}
The result follows from Proposition \ref{mstarherm} and Corollary \ref{new}.
\end{proof}

\begin{remark}
If we consider $k=m=0$ in \eqref{final} we get back to the expression in  Lemma \ref{gauss}, since $H_{0,0}(\alpha, \alpha)=1$.
\end{remark}

In the following we sum up the results of the actions of the STFT applied on different signals that we have took into consideration so far. We recall that $\varphi$ is the Gaussian function, and $h_m$ are the Hermite functions.
\begin{center}
	\begin{tabular}{| l | l |}
		\hline
		\rule[-4mm]{0mm}{1cm}
		{\bf \textbf{STFT}} & {\textbf{Results}} \\
		\hline
			\rule[-4mm]{0mm}{1cm}
		{$V_g(S^{g,x}_{n})$} & {Gabor kernel (see Theorem \ref{main})} \\
		\hline
		\rule[-4mm]{0mm}{1cm}
		{ $V_{\varphi} (S^{\varphi,x}_{n})$} & {normalized Fock kernel (see Theorem \ref{table})} \\
		\hline
		\rule[-4mm]{0mm}{1cm}
		{ $V_{h_m} (S_{n}^{h_m,x})$} & {Laguerre polynomials (see Proposition \ref{hermite1} )} \\
		\hline
		\rule[-4mm]{0mm}{1cm}
		{ $V_{h_k} (S_{n}^{h_m,x})$} & {2D-complex Hermite polynomials (see Theorem \ref{important3})} \\
		\hline
	\end{tabular}
\end{center}

Now, we aim to compute the $L^2$-norm of $V_{h_k}(S_{n}^{h_m,x})$. To achieve this goal we have to use formula \eqref{moyal}, hence we need to compute first the $L^2$-norm of the signal \eqref{sigeq}.

\begin{proposition}
\label{norm3}	
	Let $x \in \mathbb{R}$ and $a>1$. Then we have
	$$ \| S_{n}^{h_m,x} \|_{L^2(\mathbb{R})}^2= \sqrt{\pi} (-2)^m \sum_{j,k=0}^n C_k(n,a)C_j(n,a) e^{- \frac{(k-j)^2}{ n^2}} e^{\frac{2i(k-j)x}{n}} H_{m,m}\left(\frac{\sqrt{2}(k-j)}{n}, \frac{\sqrt{2}(k-j)}{n}\right).$$
\end{proposition}
\begin{proof}
By considering the definition of the superoscillations we can write the signal \eqref{sigeq} as
$$ S_n^{h_m}(t)= \sum_{j=0}^n C_j(n,a) e^{it \left(1- \frac{2j}{n}\right)} h_{m}(t-x).$$
This leads to the following
\begin{equation}
\label{norm1}
\| S_{n}^{h_m,x} \|_{L^2(\mathbb{R})}^2= \sum_{j,k=0}^n C_j(n,a)C_k(n,a) \int_{\mathbb{R}}e^{\frac{2 i t}{n}(k-j)} \left(h_m(t-x)\right)^2  dt.
\end{equation}
By using formula \eqref{important1} with $ \lambda:= \frac{2(k-j)}{n}$ we get
\begin{equation*}
\int_{\mathbb{R}}e^{\frac{2 i t}{n}(k-j)} \left(h_m(t-x)\right)^2  dt=\sqrt{\pi} (-2)^m e^{- \frac{(k-j)^2}{ n^2}} e^{\frac{2i(k-j)x}{n}} H_{m,m}\left(\frac{\sqrt{2}(k-j)}{n}, \frac{\sqrt{2}(k-j)}{n}\right).
\end{equation*}
By plugging the above formula in \eqref{norm1} we get the final result.
\end{proof}

Now, we have all the tools to compute the $L^2$-norm of $V_{h_k}[(S_{n}^{h_m,x})]$.

\begin{theorem}
\label{norm10}
Let $x \in \mathbb{R}$ and $a>1$. Then, we have
\begin{eqnarray}
\nonumber
\| V_{h_k}(S_{n}^{h_m,x}) \|_{L^2(\mathbb{R}^2)}^2&=&(-1)^m 2^{m+k}k! \pi \sum_{\ell,s=0}^n C_s(n,a)C_\ell(n,a) e^{- \frac{(s-\ell)^2}{ n^2}} e^{\frac{2i(s- \ell)x}{n}}\times\\ 
\label{h1}
&& \times H_{m,m}\left(\frac{\sqrt{2}(s- \ell)}{n}, \frac{\sqrt{2}(s- \ell)}{n}\right).
\end{eqnarray}
\end{theorem}
\begin{proof}
By formula \eqref{moyal} we have that
$$ \|  V_{h_k}(S_{n}^{h_m,x}) \|_{L^2(\mathbb{R}^2)}^2= \| h_k \|^2_{L^2(\mathbb{R})} \| S_{n}^{h_m,x}\|_{L^2(\mathbb{R})}^2.$$
Since $ \|h_k \|_{L^2(\mathbb{R})}^2=2^{k} k! \sqrt{\pi}$ and by Proposition \ref{norm3} we get the desired result.
\end{proof}
\begin{remark}
Theorem \ref{norm10} is a generalization of Theorem \ref{norm}. Indeed if we take $m=k=0$ in \eqref{h1} we get back to formula \eqref{norm8}.
\end{remark}
Now, we consider $n$ tending to infinity in the signal introduced in \eqref{sigeq}. By \eqref{limit} we get
$$ \lim_{n \to \infty} S_{n}^{h_m,x}(t,a)= e^{iax} h_{m}(t-x) =(M_{a} T_x h_{m})(t):=S_{n}^{h_m,x}(t,a).$$
In the following result we study the action of $V_{h_k}$ to the limit function $S^{h_m}(t,a)$.
\begin{proposition}
\label{limit1}
Let $a>1$ and $(x, u, \eta) \in \mathbb{R}^3$. Then for $k$, $m \geq 0$ we have

$$ V_{h_{k}}(S^{h_m,x})(u, \eta)=\sqrt{\pi} (-1)^m2^{\frac{k+m}{2}} e^{- \frac{(a-\eta)^2}{4}+ \frac{i(x+u)(a- \eta)}{2}- \frac{(x-u)^2}{4}} H_{k,m}(\alpha, \bar{\alpha}),$$

with $\alpha:=\frac{ (u-x)+i(\omega_j-\eta)}{\sqrt{2}}$.
\end{proposition}
\begin{proof}
From the definition of the short-time Fourier transform we have
$$ V_{h_{k}}(S_{n}^{h_m,x})(u, \eta)=\int_{- \infty}^{+ \infty} e^{it(a- \eta)} h_{m}(t-x)h_{k}(t-u) dt.$$
The result follows by Theorem \ref{NewResult} and by using Lemma \ref{inte2} with  $ \lambda:= a- \eta$ .
\end{proof}

\begin{theorem}\label{NRA}
Let $a>1$ and $(x, u, \eta) \in \mathbb{R}^3$. Then for $m$, $k \geq 0$ we have

$$ \lim_{n \to \infty} V_{h_k}(S_{n}^{h_m,x})(u,\eta)=\sqrt{\pi} (-1)^m2^{\frac{k+m}{2}} e^{- \frac{(a- \eta)^2}{4}+ \frac{i(x+u)(a- \eta)}{2}- \frac{(x-u)^2}{4}}H_{k,m}\left(\alpha, \overline{\alpha}\right),$$

with $\alpha:=\frac{u-x+i(a-\eta)}{\sqrt{2}}$.
\end{theorem}
\begin{proof}
Since the 2D-Hermite polynomials are entire in the variable $a$ we have that also
$$\sqrt{\pi} e^{- \frac{(a- \eta)^2}{4}+ \frac{i(x+u)(a- \eta)}{2}- \frac{(x-u)^2}{4}} H_{k,m}\left(\alpha, \overline{\alpha}\right)$$
is entire in the variable $a$. Hence the result follows by Theorem \ref{important3}, Proposition \ref{limit1} and the supershift property, see Definition \ref{supershift}.
\end{proof}

\section{Gabor frames and Zak transform}
The goal of this section is to develop a link between the theory of superoscillations and the theory of Gabor frames.
Gabor frames arise naturally in quantum mechanics, from the works of J.~von Neumann, and in information theory, from the paper of D.~Gabor, see \cite{GA}. Nowadays the theory of Gabor frames has several applications in signal processing, with significant research focusing on understanding which classes of functions can generate a Gabor frame. The general problem to characterize a Gabor frame for all functions and all parameters seems a very difficult problem. In \cite{L, SW} the problem of generating a Gabor frame is solved for Gaussian functions, in \cite{JS} for the hyperbolic secant, in \cite{GL, GL1} the problem is discussed for Hermite functions. In \cite{BKL} the Gabor frames are investigated for linear combinations of Cauchy kernels. Finally in \cite{GS} the authors proved that for totally positive functions of finite type there exists a characterization to be a Gabor frame.
\\We recall some basic notions of frames, Gabor systems, and Gabor frames, see \cite{Chr, GK2001}. 
\begin{definition}
Let {$\mathcal{H} \neq \{0\}$ be a Hilbert space. A countable family of vectors $ \{f_k\}_{k \in I}$ in $\mathcal{H}$} is  a frame for $\mathcal{H}$ if there exist constants $A$, $B>0$ such that
	$$ A \| f\| \leq \sum_{k \in I} |\langle f, f_k \rangle|^2 \leq B \| f\|, \quad \forall f \in \mathcal{H}.$$
	The postive constants $A$ and $B$ are called frame bounds.
\end{definition}
\begin{definition}[Gabor system]
	Let us fix a window function $g$. A Gabor system is a collection of functions of the form $$\lbrace{M_{m \beta}T_{n \alpha}g, \quad n,m \in \mathbb{Z}}\rbrace, \qquad \alpha, \beta >0.$$ In other terms, functions of this system can be expressed as follows
	$$(M_{m \beta}T_{n\alpha}g)(x)=e^{ i m \beta x}g(x-n\alpha).$$
\end{definition}
\begin{definition}[Gabor frame]
	A Gabor frame is a frame for $L^2(\mathbb{R})$ of the form\\
	$\mathcal{G}(g, \alpha, \beta):=\lbrace{M_{m \beta}T_{n \alpha}g, \, n,m \in \mathbb{Z}}\rbrace$ for given $\alpha$, $\beta >0$ and a fixed window function $g$ in $L^2(\mathbb{R})$.
\end{definition}
An important tool in the theory of Gabor frames is the Zak transform.
\begin{definition}[Zak transform]
We define the Zak transform $\mathcal{Z}$ of a suitable signal $f:\mathbb{R}\longrightarrow \mathbb{C}$ by the following expression
	\begin{equation}
		\mathcal{Z}(f)(u,\eta):=\sum_{k\in \mathbb{Z}}f(u-k)e^{i k \eta},
	\end{equation}
	for every $(u,\eta)\in \mathbb{R}^2$.
\end{definition}
An interesting formula of the Zak transform  is its application to the frequency shifts.
\begin{proposition}
For a fixed $(x,\omega)\in \mathbb{R}^2$ we have
	\begin{equation}
	\label{modtra}
		\mathcal{Z}(T_xM_\omega f)(u,\eta)=e^{ i\omega(u-x)}\mathcal{Z}f(u-x,\eta-\omega),
	\end{equation}
	for every $(u,\eta )\in \mathbb{R}^2$.
\end{proposition}
We note that the Zak transform of the Gaussian function $\varphi(t)=e^{-\frac{t^2}{2}}$ is given by the following expression
\begin{equation}
\label{zakgaussian}
	\displaystyle \mathcal{Z}(\varphi)(u,\eta)=\sum_{k\in \mathbb{Z}} e^{-\frac{ (u-k)^2}{2}}e^{ i k\eta},\quad \forall (u,\eta)\in \mathbb{R}^2.
\end{equation}
The Zak transform give a characterization of the existence of Gabor frames, see \cite[Cor. 8.3.2]{GK2001}.
 
\begin{theorem}
\label{gf}
Given $ \alpha>0$ and let $g \in L^2(\mathbb{R})$, $g \neq 0$. Then $ \mathcal{G} \left(g, \alpha, \frac{1}{\alpha}\right)$ is Gabor frame if and only if
$$0< a \leq | \mathcal{Z}(g) (u, \eta)| \leq b<  \infty  \quad \hbox{a.e}.$$
\end{theorem}

Now, we apply the Zak transform to the signal \eqref{sig}.

\begin{theorem}
\label{ten}
Let $a>1$ and $g \in L^2(\mathbb{R})$. For $(x, u,\eta)\in \mathbb{R}^2$ we have 
\begin{equation}
\displaystyle \mathcal{Z}(S^{g,x}_n)(u,\eta)=\sum_{j=0}^{n}C_j(n,a)e^{i \omega_j u}\mathcal{Z}(g)(u-x, \eta- \omega_j),
\end{equation}
where  $\omega_j=1-\frac{2j}{n}$.
\end{theorem}
\begin{proof}
By the commutation rule between the modulation and the translation operator, see \eqref{commu}, we get
\begin{eqnarray}
\label{tweleve}
S^{\varphi,x}_{n}(t,a)&=& F_n(t,a)(T_{x}g)(t)\\
\nonumber
&=&\sum_{j=0}^{n}C_j(n,a)(M_{\omega_j}T_xg)(t)\\
\nonumber
&=&\sum_{j=0}^n C_j(n,a) e^{i \omega_j x} (T_x M_{\omega_j} g)(t).
\end{eqnarray}	
By applying the Zak transform to \eqref{tweleve} and by formula \eqref{modtra} we have
\begin{eqnarray*}
\mathcal{Z}(S^{\varphi,x}_{n})(u, \eta)&=& \sum_{j=0}^n C_j(n,a) e^{i \omega_j x} e^{i \omega_j(u-x)} \mathcal{Z}(g)(u-x,\eta- \omega_j)\\
&=&\sum_{j=0}^n C_j(n,a) e^{i \omega_j u} \mathcal{Z}(g)(u-x,\eta- \omega_j).
\end{eqnarray*}
This proves the result.
\end{proof}

\begin{lemma}
	\label{estim0}
	Let $a>1$ and $(x, u,\eta)\in \mathbb{R}^2$ We suppose that $g \in L^2(\mathbb{R})$ and  $ g \neq 0$. Then there exists $c>0$ such that we have
	
	$$0<c\leq|\mathcal{Z}(S^{g,x}_n)(u, \eta)|$$
	
\end{lemma}
\begin{proof}
Since the functions $ e^{i k \eta}$, with $k \in \mathbb{Z}$, form an orthonormal basis for the space $L^2(\mathbb{T})$ we have that
	$$ \hbox{if} \quad \sum_{k \in \mathbb{Z}} \alpha_k e^{ i k \eta}=0 \quad \hbox{then} \quad \alpha_k=0 \quad \forall k \in \mathbb{Z}.$$
	Let us start by supposing that $ \mathcal{Z}(S^{g,x}_n)(u, \eta)=0$. By Theorem \ref{ten} we get
	$$ \mathcal{Z}(g)(u-x, \eta- \omega_j)=0$$
	However, from the definition of the Zak transform we arrive to a contradiction because $g \neq 0$. This means that there exists a constant $c>0$ such that
	$$ c\leq|\mathcal{Z}(S_{n}^{g,x})(u,\eta)|.$$
\end{proof}
Now, we give two conditions on the window function $g$ such that the function $S^{\varphi,x}_{n}$ can generate a Gabor frame. To do this we need to recall the notion of the so-called Wiener space.

\begin{definition}
A function $g \in L^{\infty}(\mathbb{R})$ belongs to the Wiener space $W(\mathbb{R})$ if
$$ \|g\|_{W}:= \sum_{n \in \mathbb{Z}} \hbox{ess sup}_{x \in Q} | g(x+n)| < \infty,$$
where $Q:= [0,1] \times [0,1]$.
\end{definition}
The relation between the Wiener space and the Zak transform is given by the following result, see \cite[Lemma 8.2.1]{GK2001}.
\begin{lemma}
\label{fourteen}
If a function $f$ belongs to the  Wiener space $W(\mathbb{R})$, then its Zak transform belongs to $L^{\infty}(\mathbb{R}^2)$.
\end{lemma}
Now, we show a condition for the function $S^{g,x}_{n}$ to be a frame.
\begin{theorem}
Given $ \alpha >0$ and $x \in \mathbb{R}$. If $g \in W(\mathbb{R})$ and $g \neq 0$, then the Gabor system $ \mathcal{G}\left(S^{g,x}_n, \alpha, \frac{1}{\alpha}\right)$ is a Gabor frame.
\end{theorem}
\begin{proof}
By hypothesis and Lemma \ref{fourteen} we have that there exist a constant $b$ such that
$$ | \mathcal{Z}(g)(u, \eta)| \leq b.$$
Since $|C_j(n,a)| \leq (1+a)^n$, by Theorem \ref{ten} we have
\begin{eqnarray}
\nonumber
| \mathcal{Z}( S^{g,x}_n)(u, \eta)| &\leq& \sum_{j=0}^n C_{j}(n,a) | \mathcal{Z}(g)(u- x, \eta- \omega_j)|\\
\label{ineq}
& \leq & (1+a)^n b.
\end{eqnarray}
By inequality \eqref{ineq} and Lemma \ref{estim0} we have that
$$ 0<c \leq |\mathcal{Z}( S^{g,x}_n)(u, \eta)| \leq (1+a)^n b,$$
where $c$ is a constant. Therefore by Theorem \ref{gf} we get that $ \mathcal{G} \left( S^{g,x}_n, \alpha, \frac{1}{\alpha}\right)$ is Gabor frame.
\end{proof}

\begin{theorem}
\label{gab}
Given $ \alpha>0$ and $x \in \mathbb{R}$. If $ \mathcal{G} \left(g, \alpha, \frac{1}{\alpha}\right)$ is a Gabor frame then we have that $ \mathcal{G} \left(S^{g,x}_n, \alpha, \frac{1}{\alpha}\right)$ is also a Gabor frame .
\end{theorem}
\begin{proof}
By hypothesis and Theorem \ref{gf} we have that there exists a positive constant $b$ such that
$$ | \mathcal{Z}(g)(u, \eta)| \leq b.$$
By similar computations performed in \eqref{ineq} and Lemma \ref{estim0} we have
$$ c \leq | \mathcal{Z}( S^{g,x}_n)(u, \eta)| \leq (1+a)^n b, \qquad c>0.$$
The result follows by using Theorem \ref{gf}.

\end{proof}

In the next two subsections, in order to get peculiar results, we are going to consider the window function $g$ being the Gaussian function and the Hermite functions.
\newline
\newline
\textbf{The Gaussian case:}
Now, we focus on the case in which the window function $g$ is the Gaussian $\varphi(t)=e^{-\frac{t^2}{2}}$ and we do not take into consideration the parameter of translation, i.e. $x=0$. Precisely, we consider the following signal
$$S^{g,0}_n(t,a)=e^{-\frac{t^2}{2}}F_n(t,a):=\tilde{F_n}(t,a),\quad \forall t\in \mathbb{R}.$$
Therefore, we can prove the following result in the case of a Gaussian window.
\begin{theorem}
\label{gauss0}
Let $a>1$. The Zak transform of the function $\tilde{F_n}(t,a)$ is given by the following expression
\begin{equation}
\mathcal{Z}(\tilde{F_n})(u,\eta)=\sum_{k\in \mathbb{Z}} \tilde{F_n}(u-k,a)e^{i k\eta}, 
\end{equation}
for every $(u,\eta)\in \mathbb{R}^2$.
\end{theorem}
\begin{proof}
By Theorem \ref{ten} and formula \eqref{zakgaussian} we have
$$ \mathcal{Z}(\tilde{F_n})(u,\eta)= \sum_{j=0}^n C_j(n,a) e^{ i \omega_j u} \left( \sum_{k \in \mathbb{Z}} e^{- \frac{(u-k)^2}{2}} e^{ i k(\eta- \omega_j)}\right).$$
By exchanging the sums, using the fact that $ \omega_j:=1- \frac{2j}{n}$ and the definition of superoscillations we obtain
\begin{eqnarray*}
\mathcal{Z}(\tilde{F}_n)(u, \eta)&= &\sum_{k \in \mathbb{Z}} e^{- \frac{(u-k)^2}{2}+i k \eta} \left( \sum_{j=0}^n C_{j}(n,a) e^{ i \omega_j(u-k)}\right)\\
&=& \sum_{k \in \mathbb{Z}} F_n(u-k,a)e^{- \frac{(u-k)^2}{2}+ i k \eta}\\
&=& \sum_{k \in \mathbb{Z}} \tilde{F}_n(u-k,a)e^{i k \eta}.
\end{eqnarray*}
This shows the result.
\end{proof}

Now our aim is to show that $\mathcal{G}\left(\tilde{F}_n,\alpha, \frac{1}{\alpha} \right)$ is a Gabor frame. To this end we recall the following notion of theta function in two variables
$$\theta(z, \tau)= \sum_{k \in \mathbb{Z}} e^{\pi i k^2 \tau+2 \pi ikz},$$
where $z \in \mathbb{C}$ and $\tau \in \mathcal{H}:= \{x+iy, \, y>0, \, x,y \in \mathbb{R} \}$, the upper half plane, see \cite{DM} for more information on the theta functions. Now, we prove some technical results.
\begin{lemma}
\label{gauss1}
Let $(u,\eta)\in \mathbb{R}^2$. Then we have the following estimate
\begin{equation}
|\mathcal{Z}(\tilde{F_n})(u,\eta)|\leq (1+a)^n e^{-\frac{u^2}{2}} \theta\left(-\frac{iu}{2 \pi}, \frac{i}{2 \pi} \right).
\end{equation}
\end{lemma}
\begin{proof}
We start by observing that
$$ | \tilde{F}_n(t,a)| \leq (1+a)^n e^{-\frac{t^2}{2}}, \qquad t \in \mathbb{R}.$$
Then by Theorem \ref{gauss0} we have
\begin{eqnarray*}
| \mathcal{Z}(\tilde{F}_n(u, \eta))| &\leq& \sum_{k \in \mathbb{Z}} | \tilde{F}_n (u-k,a)|\\
&=& (1+a)^n \sum_{k \in \mathbb{Z}} e^{- \frac{(u-k)^2}{2}}\\
&=& (1+a)^n e^{-\frac{u^2}{2}} \sum_{k \in \mathbb{Z}} e^{-\frac{k^2}{2}+ku} \\
&=& (1+a)^n e^{-\frac{u^2}{2}} \theta\left(-\frac{iu}{2 \pi}, \frac{i}{2 \pi} \right).
\end{eqnarray*}
This proves the result
\end{proof}

\begin{theorem}
Let $ \alpha>0$. The Gabor system $\mathcal{G}\left(\tilde{F}_n,\alpha, \frac{1}{\alpha} \right)$ is a Gabor frame.
\end{theorem}
\begin{proof}
By similar arguments used to show Lemma \ref{estim0} we have
$$0<c\leq|\mathcal{Z}(\tilde{F_n}(u,\eta))|.$$
Finally the result follows by Lemma \ref{gauss1} and Theorem \ref{gf}.
\end{proof}

 \section{Evolution of superoscillations}
Inspired from the results obtained in \cite{ACSST2014}
we show how to solve the auxiliary Cauchy problem for the Schrödinger equation when the initial datum is given by the translation with respect to the variable $x_0$ and the modulation with respect to the frequency shifts of a function $g$ in $L^2(\mathbb{R})$. We use this result to study the longevity of superoscillations. Precisely, we consider the following problem:
\begin{equation}
\label{Scho}
\begin{cases}
	i \frac{\partial \phi(x,t)}{\partial t}= - \frac{\partial^2 \phi(x,t)}{\partial x^2}\\
	\phi(x,0)= M_{k_0}T_{x_0}g.
\end{cases}
\end{equation}
To solve this problem we are going to use the classic Fourier transform method. By applying the Fourier transform to the equation in \eqref{Scho} we get
$$ i \frac{\partial \hat{\phi}(p,t)}{\partial t}= p^2 \hat{\phi}(p,t).$$
Integrating, we get
\begin{equation}
\label{ones}
\hat{\phi}(p,t)=C(p) e^{- i p^2 t},
\end{equation}
where the function $C(p)$ will be determined by the initial condition. By property \eqref{ts} of the Fourier transform we get 
\begin{eqnarray}
\nonumber
C(p)&=& \hat{\phi}(p,0)\\
\nonumber
&=& \mathcal{F}(M_{k_0}T_{x_0}g)\\
\label{twos}
&=& (T_{k_0} M_{-x_0} \mathcal{F}(g))(p).
\end{eqnarray}
By plugging \eqref{twos} into \eqref{ones} we have
$$ \hat{\phi}(p,t)=(T_{k_0} M_{-x_0} \mathcal{F}(g))(p)e^{-i p^2 t}.$$
By taking the inverse Fourier transform we get
\begin{eqnarray}
\nonumber
\phi(x,t)&=& \mathcal{F}^{-1} \left((T_{k_0} M_{-x_0} \mathcal{F}(g)) e^{-i p^2 t}\right)(x)\\
\label{therees}
&=& \int_{\mathbb{R}} (T_{k_0} M_{-x_0} \mathcal{F}(g))(p) e^{-i p^2 t} e^{i px} dp.
\end{eqnarray}
By Theorem \ref{intrep} we know that we can write the superoscillations by means of the following integral expression
$$ F_n(y)= \frac{1}{g(y-x) \| g\|_{L^2}^2} \int_{\mathbb{R}^2} \varphi_{n}(k_0,x_0) \phi(x, x_0, k_0) dx_0 dk_0,$$
where the function $\varphi_n$ is defined in \eqref{phi}. By evolving $\phi(x, x_0, k_0)$ in time we get that the following function
$$ F_{n}(y,t)=\frac{1}{g(y-x) \| g\|_{L^2}^2} \int_{\mathbb{R}^2} \varphi_{n}(k_0,x_0) \phi(x, t, x_0, k_0) dx_0 dk_0,$$
is also a solution of the problem \eqref{Scho}. By plugging \eqref{therees} into the above formula we get
\begin{equation}
\label{evolution1}
F_n(y,t)= \frac{1}{g(y-x)\| g\|_{L^2}^2} \int_{\mathbb{R}^3} \varphi_n(k_0, x_0) (T_{k_0} M_{-x_0} \mathcal{F}(g))(p) e^{- i p^2 t} e^{i p x} dp dk_0 dx_0.
\end{equation}
Now we give peculiar forms of the above integral representation by choosing suitable window functions $g$.
\newline
\newline 
\textbf{Hermite functions case}
\newline
We consider the case in which the function $g$ is the Hermite function. Since $ \mathcal{F}(h_k)=(-i)^k h_k$ and by formula \eqref{therees} we get
\begin{eqnarray*}
\phi(x,t, x_0, k_0)&=& \int_{\mathbb{R}} (T_{k_0}M_{-x_0} \mathcal{F}(h_k))(p) e^{-i p^2 t} e^{i p x} dp\\
&=&  (-i)^k \int_{\mathbb{R}} (T_{k_0} M_{-x_0} h_{k})(p) e^{- i p^2 t} e^{i px} dp\\
&=& (-i)^k \int_{\mathbb{R}} e^{- i x_0(p-k_0)} h_k(p-k_0) e^{- i p^2 t+ i px} dp.
\end{eqnarray*}
Now, we make the following change of variables $u=p-k_0$ and we get
\begin{eqnarray}
\nonumber
\phi(x,t, x_0, k_0)&=&  (-i)^k \int_{\mathbb{R}}  e^{- i x_0 u} h_k(u) e^{- i (u+k_0)^2 t+ i (u+k_0)x} du\\
\label{five}
&=&  (-i)^k e^{i k_0 x- ik_0^2 t} \int_{\mathbb{R}} e^{- i u^2+ iu (x- x_0-2k_0)} h_k(u) du.
\end{eqnarray}
By replacing \eqref{five} in \eqref{evolution1} we get the following integral representation
$$ F_n(y,t)= \frac{ (-i)^k }{\sqrt{\pi} h_k(y-x) 2^{\frac{k}{2}} \sqrt{k!}} \int_{\mathbb{R}^3} \varphi_n(x_0, k_0) e^{- i p^2+ i p (x- x_0-2k_0)+i k_0 x-i k_0^2t} h_k(p) dp d k_0 dx_0.$$
\newline
\newline 
\textbf{Gaussian case}
\newline
We now consider the case in which the function $g$ is the Gaussian function, i.e. $g(t)=e^{- \frac{t^{2}}{2}}$. We observe that $ (\mathcal{F}g)(\lambda)= \sqrt{2 \pi} e^{- \frac{ \lambda^2}{2}}$. By using formula \eqref{therees} we have
\begin{eqnarray*}
\phi(x,t, x_0, k_0)&=&\sqrt{2 \pi} \int_{\mathbb{R}} T_{k_0} M_{-x_0} \left(e^{- \frac{p^2}{2}}\right) e^{-i p^2 t+i px} dp\\
&=&\sqrt{2 \pi}\int_{\mathbb{R}} e^{- i x_0(p-k_0)- \frac{(p-k_0)^2}{2}}e^{-i p^2 t+i px} dp\\
&=&\sqrt{2 \pi} e^{ i x_{0}k_0- \frac{k_0^2}{2}} \int_{\mathbb{R}} e^{- \frac{p^2 (1+2it)}{2}+p (ix-ix_0+k_0)} dp.
\end{eqnarray*}
By \cite[Formula 3.323]{GR} we know that
$$ \int_{\mathbb{R}} e^{-\alpha^2 x^2+\beta x} dx= e^{\frac{\beta^2}{4 \alpha^2}} \frac{\sqrt{\pi}}{\alpha}, \qquad  \hbox{if} \quad \hbox{Re}(\alpha^2)>0.$$
By considering in the above integral $\alpha^2:=\frac{1+2it}{2}$ and $\beta:=ix-ix_0+k_0$ we get 
\begin{equation}
\label{six}
\phi(x,t,x_0, k_0)= 2 \pi \sqrt{\frac{1}{1+2it}} e^{ i x_{0}k_0- \frac{k_0^2}{2}} e^{\frac{[k_0+i(x-x_0)]^2}{2(1+2it)} }.
\end{equation}
By plugging \eqref{six} into \eqref{evolution1} we get
$$ F_{n}(y,t)= 2 \pi e^{\frac{(y-x)^2}{2}} \sqrt{\frac{1}{1+2it}} \int_{\mathbb{R}^2} \varphi_{n}(x_0, k_0) e^{ i x_{0}k_0- \frac{k_0^2}{2}} e^{\frac{[k_0+i(x-x_0)]^2}{2(1+2it)} } dx_0 dk_0.$$

 \section{STFT and approximating sequence}

In this section we investigate how the superoscillations and approximating sequences are related with the ambiguity function, Fourier transform, and the STFT. We start by recalling the notion of approximating sequence (see \cite{ACSSTbook2017}).

\begin{definition}[Approximating sequence]
Let $a>1, n\geq 1$ and $\psi\in \mathcal{S}(\mathbb{R})$. Then, the standard approximating sequence of $\psi$ is defined to be the function given by
\begin{equation}
\displaystyle \phi_{\psi, n, a}(x):=\sum_{j=0}^{n}C_j(n,a)\psi\left(x+\left(1-2j/n\right)\right),\quad x\in \mathbb{R}.
\end{equation}
\end{definition}
The approximating sequence satisfies the following integral representation, see \cite{ACSSTbook2017, ACDSS}.
\begin{lemma}[Integral representation]
\label{APSTHM}
Let $a>1, n\geq 1$ and $\psi\in \mathcal{S}(\mathbb{R})$. Then, we have
\begin{equation}
\displaystyle \phi_{\psi, n, a}(x):=\frac{1}{2\pi}\int_\mathbb{R}F_n(\lambda, a)\mathcal{F}(\psi)(\lambda) e^{i\lambda x}d\lambda,\quad x\in \mathbb{R}.
\end{equation}
The above expression can be also written as
\begin{equation}
	\mathcal{F}(\phi_{\psi, n, a})(\lambda)=\mathcal{F}(\psi)(\lambda)F_n(\lambda,a)
\end{equation}

\end{lemma}

We observe that the approximating sequence of a function $\psi\in \mathcal{S}(\mathbb{R})$ can be expressed in terms of the translation operator  applied to $\psi$:
\begin{equation}
\label{tras}
\displaystyle \phi_{\psi, n, a}(x):=\sum_{j=0}^{n}C_j(n,a)(T_{(2j/n-1)} \psi)(x),\quad x\in \mathbb{R}.
\end{equation}

So we have the following result, see \cite{ACDSS}.
\begin{theorem}
Let $\psi\in \mathcal{S}(\mathbb{R}), $ $a>1$ and $n\geq 1$. Then, we have

where $F_n(\lambda,a)$ is the classical superoscillating sequence in the frequency domain.
\end{theorem}
Now, we study the action of the Fourier transform on the approximating sequence when we consider a peculiar $\psi$.
\begin{proposition}
Let $x \in \mathbb{R}$. Fix $g \neq 0$ in $L^{2}(\mathbb{R})$. Then it holds
$$ \mathcal{F}(\phi_{T_x g, n,a})(\lambda)=M_{-x}\mathcal{F}(\phi_{g,n,a})(\lambda).$$
\end{proposition}
\begin{proof}
We start by plugging the function $\psi(t):= (T_x g)(t)$ into \eqref{tras}, and we obtain
\begin{equation}
\label{appr}
\phi_{\psi,n,a}(t)=\sum_{j=0}^n C_j(n,a) T_{\frac{2j}{n}-1}(T_x g)(t).
\end{equation}
By applying the Fourier transform to \eqref{appr} and using the first property of \eqref{basic} we get
\begin{eqnarray*}
\mathcal{F}(\phi_{\psi,n,a})(\lambda)&=& \sum_{j=0}^n C_j(n,a) \mathcal{F}\left(T_{\frac{2j}{n}-1}T_x g\right)(\lambda)\\
&=&  \sum_{j=0}^n C_j(n,a) M_{1-\frac{2j}{n}}\mathcal{F}\left(T_x g\right)(\lambda)\\
&=&  \sum_{j=0}^n C_j(n,a) M_{1-\frac{2j}{n}}M_{-x}\mathcal{F}(g)(\lambda)\\
&=&  \sum_{j=0}^n C_j(n,a)e^{i \left(1- \frac{2j}{n}\right) \lambda} e^{-ix \lambda} \mathcal{F}(g)(\lambda)\\
&=& e^{-ix \lambda} \mathcal{F}(g)(\lambda) F_n(\lambda,a).
\end{eqnarray*} 
By Lemma \ref{APSTHM} we get

$$ \mathcal{F}(\phi_{T_x g, n,a})(\lambda)=  e^{-ix \lambda}\mathcal{F}(\phi_{g,n,a})(\lambda) =M_{-x}\mathcal{F}(\phi_{g,n,a})(\lambda).$$
\end{proof}

In order to show the next results we need the following notion, see \cite{GK2001}.

\begin{definition}[Ambiguity function]
\label{ambiguity}
The ambiguity function of a generic function $\varphi \in L^2(\mathbb{R})$ is defined by
$$ A[\varphi](u, \eta)=e^{\frac{i u \eta}{2}} V_{\varphi}(\varphi)(u, \eta), \qquad \forall (u, \eta) \in \mathbb{R}^2.$$
\end{definition}

The objective of the next result is to express the STFT of the approximating function $\phi_{g,n,a}$ using the ambiguity function. Indeed we have the following result.

\begin{theorem}
Let $a>1$. Let us fix a window function $g \neq 0$ in $L^{2}(\mathbb{R})$. Then it holds that
$$ V_{g}(\phi_{g,n,a})(u, \eta)=e^{- \frac{i u \eta}{2}} \sum_{j=0}^n C_{j}(n,a) e^{\frac{i \eta}{2}\omega_j} A[g] \left(u+\omega_j, \eta\right), \qquad \forall (u, \eta) \in \mathbb{R}^2,$$
where $ \omega_j=1-\frac{2j}{n}$.
\end{theorem}
\begin{proof}
By applying the STFT to the approximating sequence \eqref{tras} and by using the so-called covariance property, see \eqref{cov}, we have
\begin{eqnarray*}
\nonumber
V_{g}(\phi_{g,n,a})(u, \eta)&=& \sum_{j=0}^n C_{j}(n,a)V_g\left(T_{\frac{2j}{n}-1}(g)\right)(u, \eta)\\
\label{appST}
&=& \sum_{j=0}^n C_{j}(n,a)e^{i \omega_j \eta}V_g(g)\left(u+\omega_j, \eta\right).
\end{eqnarray*}
Now from  Definition \ref{ambiguity} we have
\begin{eqnarray*}
V_{g}(\phi_{g,n,a})(u, \eta)&=&\sum_{j=0}^n C_{j}(n,a)e^{i \omega_j \eta} e^{- \frac{i \eta}{2} \left(u+\omega_j \right)}A[g]\left(u+\omega_j, \eta\right)\\
&=& e^{- \frac{i u \eta}{2}} \sum_{j=0}^n C_{j}(n,a) e^{\frac{i \eta}{2} \omega_j} A[g] \left(u+\omega_j, \eta\right).
\end{eqnarray*}

\end{proof}

Now, we want to compute the action of the STFT on the approximating sequence $\phi_{\psi, n,a}$ with $\psi$ being the Hermite functions. We will find a relation with the complex-Hermite polynomials, introduced in \eqref{compexhermite}.
\begin{theorem}
\label{last}
Let $a>1$, $k$, $m \geq 0$ and $ \omega_j=1- \frac{2j}{n}$ . Then it holds that
$$ V_{h_k}(\phi_{h_m,n,a})(u, \eta)=\sqrt{\frac{\pi}{k!}}2^{\frac{k}{2}} e^{-\frac{i u \eta}{2}- \frac{u^2+\eta^2}{4}} \sum_{j=0}^n C_j(n,a) e^{- \frac{1}{4}\omega_j^2-\frac{1}{2}(u-i\eta)\omega_j}H_{k,m}(z_j, \bar{z_j}),$$
with $z_j:= \frac{ \left(u+\omega_j\right)+i \eta}{\sqrt{2}}$.
\end{theorem}
\begin{proof}
By the definition of approximating sequence, see \eqref{tras}, and the covariance property, see \eqref{cov}, we have
\begin{eqnarray}
\nonumber
V_{h_k}(\phi_{h_m,n,a})(u, \eta)&=& \sum_{j=0}^{n}C_j(n,a) V_{h_k} \left(T_{\frac{2j}{n}-1}(h_m)\right)(u, \eta)\\
\label{aprrox}
&=& \sum_{j=0}^{n}C_j(n,a) e^{ i  \omega_j\eta} V_{h_k}(h_m) \left(u+\omega_j, \eta \right).
\end{eqnarray}
By \cite{Asampling} we know that
\begin{equation}
\label{ah}
V_{h_k}(h_m) \left( \bar{z}\right)= \sqrt{\frac{\pi}{k!}}2^{\frac{k}{2}}e^{\frac{i u \eta}{2}- \frac{|z|^2}{2}} H_{k,m}(z, \bar{z}), \qquad z=\frac{u+i \eta}{\sqrt{2}}.
\end{equation}
By plugging \eqref{ah} into \eqref{aprrox} we obtain
\begin{eqnarray*}
 V_{h_k}(\phi_{h_m,n,a})(u, \eta)&=& \sqrt{\frac{\pi}{k!}}2^{\frac{k}{2}} \sum_{j=0}^{n}C_j(n,a) e^{ i \omega_j\eta} e^{-\frac{i \eta}{2} \left(u+\omega_j\right) } e^{- \frac{1}{4} \left[\left(u+\omega_j\right)^2+ \eta^2\right]} H_{k,m}(z_j, \bar{z_j})\\
 &=& \sqrt{\frac{\pi}{k!}}2^{\frac{k}{2}} e^{-\frac{i u \eta}{2}- \frac{1}{4} \eta^2} \sum_{j=0}^n C_j(n,a) e^{ \frac{i \eta \omega_j}{2}} e^{- \frac{1}{4} \left(u+\omega_j\right)^2} H_{k,m}(z_j, \bar{z_j})\\
 &=& \sqrt{\frac{\pi}{k!}} 2^{\frac{k}{2}}e^{-\frac{i u \eta}{2}- \frac{u^2+\eta^2}{4}} \sum_{j=0}^n C_j(n,a) e^{- \frac{1}{4}\omega_j^2-\frac{1}{2}(u-i\eta)\omega_j}H_{k,m}(z_j, \bar{z_j}).
\end{eqnarray*}
\end{proof}

Now, by using the operator $\mathcal{M}_{z}$ defined in \eqref{ope} we can write the result of Theorem \ref{last} in the following way.

\begin{proposition}
\label{app}
Let $a>1$,  $k$, $m \geq 0$ and $ \omega_j=1- \frac{2j}{n}$  then it holds that
$$ V_{h_k}(\phi_{h_m,n,a})(u, \eta)=\sqrt{\frac{\pi}{k!}}\mathcal{M}_{\bar{p}}^{-1}\left(\sum_{j=0}^n C_j(n,a) e^{- \frac{1}{4}\omega_j^2-\frac{1}{2}(u-i\eta)\omega_j}H_{k,m}(z_j, \bar{z_j}) \right),$$
where  $z_j:= \frac{ \left(u+\omega_j\right)+i \eta}{\sqrt{2}}$ and $p=u+i \eta$.
\end{proposition} 
\begin{proof}
The result follows from the definition of the operator $M_{\bar{p}}$, which is defined as $ M_{\bar{p}}^{-1}=  e^{-\frac{i u \eta}{2}- \frac{u^2+\eta^2}{4}}.$
\end{proof}

Now, if we take $m=k=0$ in the definition of Hermite function we have that $h_m(t)=h_k(t)= e^{- \frac{t^2}{2}}$ and the 2D-complex Hermite polynomials become $H_{0,0}(z,w)=1$. The following result follows easily from Proposition \ref{app}.

\begin{corollary}
\label{app1}
Let $a>1$ and $ \varphi(t):= e^{- \frac{t^2}{2}}$. For $ \omega_j=1- \frac{2j}{n}$ we have
$$ V_{\varphi}(\phi_{\varphi,n,a})(u, \eta)=\sqrt{\pi}\mathcal{M}_{\bar{p}}^{-1}\left(\sum_{j=0}^n C_j(n,a) e^{- \frac{1}{4}\omega_j^2-\frac{1}{2}(u-i\eta)\omega_j} \right),$$
where  $p=u+i \eta$.
\end{corollary}

By \cite[Theorem 3.3.4]{ACSSTbook2017} we know that

$$ \lim_{n \to \infty} \phi_{\psi,n,a}(x)=\psi(x+a):= \phi_{\psi,a}, \qquad \psi \in \mathcal{S}(\mathbb{R})$$

\begin{proposition}
\label{app2}
Let $a>1$ and $ \varphi(t):= e^{- \frac{t^2}{2}}$. Then we have that
$$ V_{\varphi}(\phi_{\varphi,a})(u, \eta)=\sqrt{\pi}e^{- \frac{u^2+ \eta^2+a^2}{4}} e^{-\frac{(u-i \eta)a}{2}} e^{-\frac{iu \eta}{2}}, \qquad (u, \eta) \in \mathbb{R}^2.$$
\end{proposition}
\begin{proof}
From the definition of the STFT we have
\begin{eqnarray*}
V_{\varphi}(\phi_{\varphi,a})(u, \eta)&=&\int_{\mathbb{R}}\varphi(t-x) \varphi(t+a) e^{-i \eta t} dt\\
&=& \int_{\mathbb{R}} e^{-\frac{(t-u)^2}{2}} e^{- \frac{(t+a)^2}{2}} e^{-i \eta t}dt\\
&=& e^{- \frac{u^2+a^2}{2}} \int_{\mathbb{R}} e^{-t^2+t(u-a-i \eta)} dt.
\end{eqnarray*}
By the Gaussian integral, see \eqref{gint}, we have
\begin{eqnarray*}
V_{\varphi}(\phi_{\varphi,a})(u, \eta)&=& \sqrt{\pi}e^{- \frac{u^2+a^2}{2}} e^{\frac{(u-a-i \eta)^2}{4}}\\
&=& \sqrt{\pi}e^{- \frac{u^2+ \eta^2+a^2}{4}} e^{-\frac{(u-i \eta)a}{2}} e^{-\frac{iu \eta}{2}}.
\end{eqnarray*}
\end{proof}

\begin{proposition}
Let $a>1$. Then we have
$$ \lim_{n \to \infty} V_{\varphi} (\phi_{\varphi,n,a})(u,\eta)=\sqrt{\pi}e^{- \frac{u^2+ \eta^2+a^2}{4}} e^{-\frac{(u-i \eta)a}{2}} e^{-\frac{iu \eta}{2}}, \qquad (u, \eta) \in \mathbb{R}^2.$$
\end{proposition}
\begin{proof}
We start by observing that the function 
$$\sqrt{\pi}e^{- \frac{u^2+ \eta^2+a^2}{4}} e^{-\frac{(u-i \eta)a}{2}} e^{-\frac{iu \eta}{2}}$$
is entire in the variable $a$. Hence by the supershift property, Corollary \ref{app1} and Proposition \ref{app2} we get the result.
\end{proof}

\textbf{Acknowledgments:}
Daniel Alpay thanks the Foster G. and Mary McGaw Professorship in Mathematical Sciences, which supported this research. Daniele C. Struppa is grateful to the Donald Bren Presidential Chair in Mathematics. Antonino De Martino was supported by MUR grant “Dipartimento di Eccellenza 2023-2027". Kamal Diki thanks the Grand Challenges Initiative (GCI) at Chapman University, which supported this research.
The authors are grateful to the anonymous referees whose deep and extensive comments greatly contributed to improve this paper. The authors are grateful to Prof.Ahmed Sebbar for pointing out \cite[Formula 2.0.1]{F}, which was a crucial step in the proof of Lemma \ref{inte2} and its  consequences.

\end{document}